\documentclass[reqno]{amsart}

\usepackage{amsthm, amsmath, amssymb, amsfonts, graphicx}

\usepackage[english]{babel}

\usepackage{bbm}
\usepackage[foot]{amsaddr}

\usepackage[inline]{enumitem}

\usepackage{float}
\restylefloat{table}

\usepackage{graphicx}
\usepackage[font=sl,labelfont=bf]{caption}
\usepackage{subcaption}

\usepackage[colorlinks=true, pdfstartview=FitV, linkcolor=blue,
            citecolor=blue, urlcolor=blue]{hyperref}
\usepackage[usenames]{color} 

\usepackage[round]{natbib}

\newtheorem{theorem}{Theorem}
\newtheorem{proposition}[theorem]{Proposition}
\newtheorem{lemma}[theorem]{Lemma}
\newtheorem{corollary}[theorem]{Corollary}

\theoremstyle{definition}

\newtheorem{example}[theorem]{Example}
\theoremstyle{remark}
\newtheorem{remark}[theorem]{Remark}

\numberwithin{equation}{section}
\numberwithin{theorem}{section}

\DeclareMathOperator*{\argmax}{arg\,max} 
\DeclareMathOperator*{\Var}{Var} 

\def\Bbox{\mathfrak{B}}

\DeclareMathOperator{\spann}{span}

\title[Long-run risk-sensitive  portfolio optimisation]{Long-run risk-sensitive  portfolio optimisation with proportional transaction costs and log L\'{e}vy asset prices}
\author{Damian Jelito$^{\dagger}$}

\author{\L{}ukasz Stettner$^{\ast}$}
\address{$^{\dagger}$ Jagiellonian University, Institute of Mathematics, Krak\'o{w}, Poland}
\address{$^{\ast}$Institute of Mathematics Polish Academy of Sciences, Warsaw, Poland}
\thanks{$^{\dagger}$Corresponding author}
\email{damian.jelito@uj.edu.pl}
\email{stettner@impan.pl}

\makeatletter
\def\namedlabel#1#2{\begingroup
    #2%
    \def\@currentlabel{#2}%
    \phantomsection\label{#1}\endgroup
}
\makeatother

\begin{document}

\begin{abstract}
We study a long-run risk-sensitive portfolio problem with proportional transaction costs in a continuous-time market whose log-prices are given as a L\'{e}vy process, and rebalancing is possible only at random moments of investment opportunities. Using a Schauder fixed-point argument, we solve the ergodic Bellman equation without any mixing assumptions. Under additional full-support and growth conditions, the solution to the Bellman equation is unique, with the unique continuous maximiser. We also prove vanishing risk-aversion asymptotics towards the risk-neutral (Kelly) problem and convergence of a dyadic-time-grid approximation of the random intervention moments. Numerical examples illustrate the results.

\bigskip
\noindent \textbf{Keywords:} risk-sensitive portfolio optimisation; proportional transaction costs; Kelly criterion; ergodic Bellman equation; L\'{e}vy process

\bigskip
\noindent \textbf{MSC2020 subject classifications:}  91G10, 93E20, 49L20, 90C40
\end{abstract}

\maketitle


\section{Introduction}\label{S:intro}

In this paper, we consider a long-run risk-sensitive portfolio optimisation problem with proportional transaction costs, in which the portfolio composition can be modified only at some externally given random moments of investment opportunities. More specifically, denoting by $W(t-)$ a controlled portfolio value at time $t$ (before transactions), we aim to find an investment strategy which maximises
\begin{equation}\label{eq:mu_log_return}
    \liminf_{n\to\infty} \frac{1}{\mathbb{E}[\tau_n]}\mu^\gamma\left[\ln \left(W(\tau_n-)/W(0-)\right) \right],
\end{equation}
where $\tau_n$ is the (random) time of the $n$th portfolio change.
Also, $\mu^\gamma$ is the risk-sensitive optimality functional (also referred to as the entropic utility function) given, for a sufficiently integrable random variable $Z$ and a risk-aversion parameter $\gamma\leq 0$, as
\begin{equation}\label{eq:entropic_utility}
    \mu^{\gamma}[Z]:=\begin{cases}
    \frac{1}{\gamma}\ln \mathbb{E}[\exp\left(\gamma Z\right)], & \gamma< 0,\\
    \mathbb{E}[Z], & \gamma=0.
    \end{cases}
\end{equation}
In this functional, $\gamma=0$ corresponds to the risk-neutral preferences while $\gamma<0$ reflects the risk-averse behaviour of a decision maker. It should be noted that, for $\gamma=0$, maximisation of~\eqref{eq:mu_log_return} corresponds to the maximisation of the expected logarithmic rate of return (per unit time) and could be linked to the celebrated Kelly criterion; see~\cite{MacLeaThiZie2011} for an overview. Also, it can be shown that for $\gamma<0$ close to $0$, one has $\mu^{\gamma}[Z]\approx \mathbb{E}[Z]+\frac{\gamma}{2}\Var[Z]$, which highlights that the maximisation of $\mu^\gamma$ could be linked to the maximisation of the expectation with a penalty term related to the variance of the results; this explains why the risk-sensitive framework is often seen as a time-consistent counterpart of the classical mean-variance approach of~\cite{Mar1952}, see~\cite{BiePli2003} for a detailed discussion of the economic properties of the criterion~\eqref{eq:mu_log_return}. Let us also recall that the entropic utility is essentially the unique family of cash-additive and time-consistent certainty equivalents, see~\cite{KupSch2009}, and that it constitutes an acceptability index in the sense of~\cite{CheMad2009}; we refer to~\cite{FolSch2016} for a general treatment of the underlying risk-measurement framework. 

The market model considered here has two features that jointly determine the mathematical structure of the problem. First, the asset prices evolve in \emph{continuous time}: we assume that the process $\ln S(t)$, $t\geq 0$, has stationary and independent increments, which covers, in particular, the exponential L\'evy dynamics, see~\cite{App2009}. Second, the investor is \emph{not} allowed to trade continuously. Instead, the portfolio can be rebalanced only at the moments $\tau_0=0<\tau_1<\tau_2<\ldots$ of an exogenously given sequence of investment opportunities, with the waiting times $\eta_i:=\tau_i-\tau_{i-1}$ satisfying $\mathbb{P}[\eta_i>0]=1$ and $\mathbb{E}[\eta_i]<\infty$. This reflects a number of practically relevant frictions: illiquidity, limited access to the market, execution delays, or simply the fact that trading decisions are made at discrete (and possibly irregular) inspection times.  
Together with the proportional transaction costs paid at each rebalancing, this makes the problem a long-run impulse-type control problem for the portfolio weight process.

Risk-sensitive portfolio optimisation has been extensively studied in the literature; we refer to~\cite{DacLle2014} and~\cite{Pri2007} for a general overview. In the continuous-time frictionless setting, the seminal contributions are due to~\cite{BiePli1999} and~\cite{FleShe2000}, and further developments, including problems with factors and with jumps, can be found e.g. in~\cite{Nag2003,Nag2006,PitSte2016,BoLiaYu2022}. On the abstract level, the criterion~\eqref{eq:mu_log_return} belongs to the class of ergodic (average cost per unit time) risk-sensitive control problems; see~\cite{BisBor2023} for a recent survey and~\cite{DiMSte1999,DiMSte2007,Jas2007,BorMey2002,GuaRob2012,BauRie2014} for a selection of general results in the Markov decision process framework; see also the monograph~\cite{BauRie2011}. On the technical side, three approaches dominate the literature: the \emph{vanishing discount} approach, see e.g.~\cite{CavCad2009,CavHer2017}; the \emph{span-contraction} approach, see e.g.~\cite{PitSte2016,SheStaObe2013}, which typically relies on a uniform ergodicity or minorisation condition of Doeblin type, cf.~\cite{HaiMat2011}; and \emph{Krein--Rutman}-type or Collatz--Wielandt-type arguments applied to the multiplicative form of the Bellman equation, see~\cite{Bon1962} and~\cite{AraBorKum2016}. Each of these techniques requires some form of ergodic behaviour of the controlled process, an assumption that is restrictive in the transaction cost setting, where the controlled weight process may be pushed towards the boundary of the simplex.

Portfolio problems with proportional transaction costs are, by now, classical, but the combination with the risk-sensitive long-run criterion is considerably less explored. In the continuous-time setting we refer to~\cite{BiePli2000} and~\cite{Ste2011b}, while~\cite{Kor1999} discusses the impulse control formulation of financial problems with fixed and proportional costs. The long-run impulse control framework with random intervention times, which is close in spirit to the present setting, has been studied in~\cite{PalSte2017,JelPitSte2019b,PitSte2019,Ste2022,JelSte2023}; see also~\cite{JelPitSte2019a,JelSte2020} for the related risk-sensitive stopping problems. The closest reference to this paper is, however, the discrete-time study of~\cite{PitSte2023}, which we now discuss in more detail.

In~\cite{PitSte2023}, the authors consider a discrete-time market with i.i.d. log-returns, proportional transaction costs, and no short selling, and they solve the associated risk-sensitive Bellman equation under the minimal integrability condition requiring finite mean and finite entropy of the log-returns. For the risk-seeking case $\gamma>0$, the Bellman equation is solved directly via the vanishing discount approach. For the risk-averse case $\gamma<0$, which is the focus of the present paper, the vanishing discount argument yields only a \emph{recursive} (iterated) family of Bellman equations, and in order to solve a single ergodic Bellman equation the authors additionally invoke a span-contraction argument, which requires an extra ergodicity (mixing) assumption on the returns as well as an auxiliary diversification constraint that is subsequently relaxed. The present paper complements and extends~\cite{PitSte2023} in the following four directions.

\begin{enumerate}[label=(\arabic*),leftmargin=*]
\item \textbf{Continuous time and random intervention moments.} We work with a continuous-time price process with stationary independent increments and allow the portfolio to be rebalanced only at random moments $(\tau_n)$ of investment opportunities. This is a genuine generalisation of the discrete-time framework of~\cite{PitSte2023}, which corresponds to $\eta_i\equiv 1$; in particular, it covers exponential L\'evy dynamics observed at renewal-type or externally driven inspection times. The resulting Bellman equation~\eqref{eq:Bellman_log} is stated at the (subordinated) growth vector $R(\eta_1)$ and its normalisation is expressed per unit of time.
\item \textbf{No ergodicity condition.} We show in Theorem~\ref{th:existence_concave} that the ergodic Bellman equation admits a solution under the increments condition~\eqref{B:increments}, the waiting time condition~\eqref{B:waiting_times}, and the integrability condition~\eqref{A:integrability_log} alone. In particular, no mixing, minorisation, or full-support assumption is required, no auxiliary diversification constraint is introduced, and no recursive scheme is needed: the Bellman equation is solved directly. Moreover, the solution produced is regular, in the sense that its span seminorm and its Lipschitz constant are bounded by explicit constants depending only on the transaction cost rates, and it is \emph{structurally distinguished}: the associated positively homogeneous lift on the cone of capital vectors is concave, equivalently, $e^{\bar{v}^\gamma}$ is concave on the simplex. To the best of our knowledge, this concavity property is new even in the discrete-time i.i.d.\ setting of~\cite{PitSte2023}.
\item \textbf{Uniqueness of the maximiser and the solution to the Bellman equation.} We prove in Theorem~\ref{th:unique_maximiser} that, for a solution with a concave lift, the supremum on the right-hand side of the Bellman equation is attained at \emph{exactly one} point, for every value of the pre-rebalancing weights, and that the resulting map is continuous. Consequently, for a fixed solution to the Bellman equation, the optimal rebalancing is unambiguously defined and the optimal stationary strategy from Theorem~\ref{th:verification} requires no measurable selection argument. This is where the full-support condition~\eqref{A:support} enters; let us stress that, in contrast to~\cite{PitSte2023}, where an ergodicity-type condition is needed for the \emph{existence} of a solution, here the non-degeneracy of the returns is used only for the \emph{uniqueness of the maximiser}. Also, it should be noted that the uniqueness is not a simple conclusion of the concavity of the related Bellman equation. In fact, as shown in Example~\ref{ex:non_concave}, part of the equation responsible for transaction costs need not to be concave, and the additional change-of-geometry argument is needed. Finally, using the full-support and growth conditions~\eqref{A:support} and~\eqref{A:growth_gamma}, in Theorem~\ref{th:uniqueness}, we show that the solution to the Bellman equation is also unique. The argument here uses the observation that, under the growth condition, every asset is present in the optimal portfolio.
\item \textbf{Vanishing risk-aversion asymptotics and discrete time grid approximation.} We show that the optimal value of the risk-sensitive problem converges, as $\gamma\to 0$, to the optimal value of the risk-neutral (Kelly-type) problem, and that the corresponding risk-neutral strategy is asymptotically optimal for the risk-sensitive criterion; see Theorem~\ref{th:gamma_asympt}, Corollary~\ref{cor:gamma_asympt_value}, Theorem~\ref{th:strategy_gamma_asympt}, and Theorem~\ref{th:gamma_asympt_strategy}. This justifies approximating a complex risk-averse problem by a simpler risk-neutral one when the risk-aversion parameter is small. Furthermore, in Theorem~\ref{th:discrete_convergence} we show that the optimal value obtained when the intervention times are restricted to a dyadic time grid converges, as the mesh tends to zero, to the optimal value of the original problem. This is important from the numerical perspective, as it requires simulating the controlled trajectory only at finitely many grid points.
\end{enumerate}

The paper is organised as follows. In Section~\ref{S:preliminaries} we introduce the market model, state the standing assumptions, formulate the associated Bellman equation, and prove the verification theorem which links its solutions to the optimal value and to the optimal strategy of the underlying control problem. Section~\ref{S:Bellman} contains the main results of the paper: in Section~\ref{S:cone} we develop the capital coordinates and study the self-financing cone, in Section~\ref{S:reward} we collect the properties of the entropic utility and of the one-period reward functional, in Section~\ref{S:existence} we prove the existence of a solution with a concave lift, in Section~\ref{S:unique_maximiser_subsec} we show the uniqueness and the continuity of the maximiser, and in Section~\ref{SS:unique_solution} we prove the uniqueness of the solution to the Bellman equation. Section~\ref{S:asymptotics} is devoted to the vanishing risk-aversion asymptotics, while in Section~\ref{S:discrete} we discuss the discrete time grid approximation. Finally, in Section~\ref{S:numerical} we present exemplary setting where all the assumptions are satisfied and illustrate the theoretical results with numerical examples.


\section{Preliminaries}\label{S:preliminaries}

In this section, we give more details on the problem formulation. We start by introducing some notation. By $\mathbb{N}$, $\mathbb{N}_{*}$, and $\mathbb{R}_+$ we denote the sets of non-negative integers, positive integers, and non-negative reals, respectively. Next, we write $\mathbf{0}$ and $\mathbf{1}$ for $d$-dimensional vectors of zeroes and ones, respectively. We write $\mathbb{R}^d_+$ for the non-negative
orthant and set $\mathbb{K}:=\mathbb{R}^d_+\setminus\{\mathbf{0}\}$. Also, for any vectors $x,y\in \mathbb{R}^d$, $d\in \mathbb{N}_*$, by $\langle x, y\rangle \in \mathbb{R}$, $x\cdot y\in \mathbb{R}^d$, and $\Vert x\Vert_1 :=\sum_{i=1}^d|x_i|\in \mathbb{R}_+$ we denote their (standard) scalar product, component-wise product, and $\ell^1$-norm, respectively.  
Next, for a set $A\subset \mathbb{R}^d$, by $\mathcal{B}_b(A)$ and $\mathcal{C}_b(A)$, we denote the families of bounded measurable real-valued functions on $A$ and continuous bounded real-valued functions on $A$, respectively. We equip these spaces with the supremum norm $\Vert f\Vert_A:=\sup_{x\in A}|f(x)|$. Finally, for $f\in \mathcal{B}_b(A)$, we define the span seminorm $\spann_A f:=\sup_{x\in A} f(x)-\inf_{x\in A} f(x)$. In the following, we often omit the subscript $A$ when the domain is clear.

Let $(\Omega, \mathcal{F}, (\mathcal{F}_t)_{t\geq 0},\mathbb{P})$ be a continuous-time filtered probability space. We consider a market with $d\in \mathbb{N}_{*}$ assets and by $S(t):=(S_1(t),\ldots, S_d(t))$ we denote the vector of their prices at time $t\geq 0$; we assume that the prices are always positive. Also, by $r(t_1, t_2):=(\ln \frac{S_1(t_2)}{S_1(t_1)}, \ldots, \ln \frac{S_d(t_2)}{S_d(t_1)})$, and 
\begin{equation}\label{eq:R(t)_fp}
    R(t_1,t_2):=e^{r(t_1,t_2)}:=\left( \frac{S_1(t_2)}{S_1(t_1)}, \ldots,  \frac{S_d(t_2)}{S_d(t_1)}\right), \quad 0\leq t_1\leq t_2<\infty,
\end{equation} 
we denote the corresponding vectors of the logarithmic rate of returns and growth rates for the time interval $[t_1,t_2]$, respectively. In this paper, we assume that the increments of the process $\ln S(t)$, $t\geq 0$, are stationary and independent. Also, for brevity, we set $R(t):=R(0,t)$ and $r(t):=\ln R(t)$, $t\geq 0$.

Next, by $N(t):=(N_1(t), \ldots, N_d(t))$, $t\geq 0$, we denote the number of units of assets held in a portfolio at time $t$ after the rebalancing is executed; we assume that short-selling is not allowed. With this terminology, we define the total portfolio value at time $t$ before and after the rebalancing by
\begin{align}\label{eq:wealth_process}
    W(t-):=\langle N(t-),S(t)\rangle, \quad \text{and} \quad    W(t):=\langle N(t),S(t)\rangle,
\end{align}
respectively. With each transaction, we associate proportional transaction costs. More specifically, we assume that an investor updating the portfolio pays the proportional cost given as
\begin{equation}\label{eq:proportional_cost}
  p(x):=\langle c,[x]^+\rangle+\langle h,[x]^-\rangle=\sum_{i=1}^d \max(c_i x_i, -h_ix_i),  
\end{equation}
where $x=(x_1, \ldots, x_d)\in \mathbb{R}^d$ reflects the changes of capital invested in each of $d$ assets, $[x]^+$ and $[x]^-$ denotes the component-wise positive and negative part of $x$, while $c:=(c_1, \ldots, c_d)$ and $h:=(h_1, \ldots, h_d)$ are the proportional commissions satisfying $c_i, h_i\in [0,1)$ for any $i=1,\ldots, d$. 
Thus, we assume the following self-financing condition
\begin{align}\label{eq:self_financing_fp}
    W(t)=W(t-)-p((N(t)-N(t-))\cdot S(t)), \quad t\geq 0.
\end{align}

In the following, instead of choosing directly the number of units $N$, we will control the fractions of capital invested in each asset. To do this, we define the families of admissible portfolio weights and strictly positive portfolio weights by
\begin{align*}
    \mathcal{S}&:=\{\pi=(\pi_1, \ldots, \pi_d)\in \mathbb{R}^d\colon \langle \pi,\mathbf{1}\rangle =1,\, \pi_i\geq 0, \, i=1, \ldots, d\}\\
    \mathcal{S}_0&:=\{\pi=(\pi_1, \ldots, \pi_d)\in \mathcal{S}\colon \pi_i>0, i=1, \ldots, d\}.
\end{align*}
Also, we define the vectors of asset weights before and after the rebalancing given by
\begin{align}\label{eq:weights_fp}
    \pi(t-):=\frac{N(t-)\cdot S(t)}{W(t-)} \quad \text{and}\quad \pi(t):=\frac{N(t)\cdot S(t)}{W(t)}, 
\end{align}
respectively. Then, by the positive homogeneity of $p$, we get 
\begin{align*}
  p((N(t)-N(t-))\cdot S(t)) & = W(t-)p\left(\frac{N(t)\cdot S(t)}{W(t)}\frac{W(t)}{W(t-)}-\frac{N(t-)\cdot S(t)}{W(t-)}\right) \\
  &= W(t-)p\left(\pi(t)\frac{W(t)}{W(t-)}-\pi(t-)\right).  
\end{align*}
Thus, we note that the self-financing condition~\eqref{eq:self_financing_fp} is equivalent to
\begin{equation}\label{eq:self_financing_weights2}
    W(t)=W(t-)-W(t-)p\left(\pi(t)\frac{W(t)}{W(t-)}-\pi(t-)\right), \quad t\geq 0.
\end{equation}
In the following, we will extensively use the observation that the ratio $W(t)/W(t-)$ could be determined directly from the portfolio weights before and after rebalancing. More specifically, using Lemma 2.1 in~\cite{PitSte2023} we may find a continuous map $s\colon \mathcal{S}\times \mathcal{S}\to [\bar{s},1]$ with some $\bar{s}\in (0,1]$ such that
\begin{equation}\label{eq:wealth_ratio_transactions}
    W(t)/W(t-) = s(\pi(t-),\pi(t)), \quad t \geq 0.
\end{equation}
 Also, for any $0\leq t_1\leq t_2$, assuming that there are no transactions in the time interval $(t_1,t_2]$, we have
\begin{multline*}
\pi(t_2):=\frac{N(t_2)\cdot S(t_2)}{\langle N(t_2),S(t_2)\rangle} = \frac{N(t_1)\cdot S(t_2)}{\langle N(t_1),S(t_2)\rangle} \\
=\frac{\frac{N(t_1)\cdot S(t_1)}{\langle N(t_1),S(t_1)\rangle}\cdot R(t_1,t_2)}{\langle \frac{N(t_1)}{\langle N(t_1),S(t_1)\rangle},R(t_1,t_2)\rangle}  = \frac{\pi(t_1)\cdot R(t_1,t_2)}{\langle \pi(t_1),R(t_1,t_2)\rangle}=G(\pi(t_1),R(t_1,t_2)), 
\end{multline*}
where $G(x,y):=\frac{x\cdot y}{\langle x,y\rangle}$, $x,y\in \mathbb{R}^d$. 

In this paper, we assume that a decision maker can change the portfolio composition only at some externally given random moments of investment opportunities (yet, there is no obligation for such a change). More specifically, we assume that we are given a sequence of independent (and independent of the price process), identically distributed random variables $(\eta_i)_{i=1}^\infty$ with values in $(0,\infty)$. Then, we define recursively
\[
\tau_0:=0, \quad \tau_{n+1}=\tau_n+\eta_{n+1}, \quad n\in \mathbb{N}.
\]
With this notation, the portfolio composition could be changed only at $\tau_n$, $n\in \mathbb{N}$.
In the following, for any $\pi \in \mathcal{S}$, by $\mathcal{A}(\pi)$ we denote the family of admissible (i.e., satisfying the self-financing and the intervention time conditions) portfolio weight strategies $\mathbf\Pi=(\pi(t))$ with $\pi(0-)=\pi$. 

Now, we reformulate the functional~\eqref{eq:mu_log_return}  using the structure defined above. Recalling~\eqref{eq:wealth_ratio_transactions}, for any $\gamma\leq 0$, $\pi \in \mathcal{S}$ and $\Pi \in \mathcal{A}(\pi)$, the functional from~\eqref{eq:mu_log_return} could be identified with
\begin{align}\label{eq:J(pi)_log}
    J^\gamma(\mathbf\Pi,\pi)&:=\liminf_{n\to\infty}\frac{1}{\mathbb{E}[\tau_n]}\mu^\gamma \left[\sum_{i=0}^{n-1} \ln \frac{W(\tau_{i+1}^-)}{W(\tau_i^-)}\right]\nonumber\\
    &=\liminf_{n\to\infty}\frac{1}{\mathbb{E}[\tau_n]}\mu^\gamma \left[\sum_{i=0}^{n-1} \ln \frac{W(\tau_{i})}{W(\tau_i^-)}+\ln \frac{W(\tau_{i+1}^-)}{W(\tau_i)}\right]\nonumber\\
    &=\liminf_{n\to\infty}\frac{1}{\mathbb{E}[\tau_n]}\mu^\gamma \left[\sum_{i=0}^{n-1}\left( \ln s(\pi(\tau_i^-),\pi(\tau_i))+\ln \langle \pi(\tau_i),R(\tau_i,\tau_{i+1})\rangle\right)\right].
\end{align}
In the following, we will maximise this functional over $\mathbf\Pi\in \mathcal{A}(\pi)$.

To obtain a solution to the problem, we assume the following conditions
\begin{enumerate}
\item[(\namedlabel{B:increments}{$\mathcal{A}1$})]  (Increments). The process $\ln S(t)$, $t\in [0,\infty)$, is right-continuous and has stationary independent increments. 
\item[(\namedlabel{B:waiting_times}{$\mathcal{A}2$})]   (Waiting times) For any $i\in \mathbb{N}_*$, we have $\mathbb{P}[\eta_i>0]=1$ and $\mathbb{E}[\eta_i]<\infty$.
\item[(\namedlabel{A:integrability_log}{$\mathcal{A}3$})] (Integrability).  For any $\gamma<0$, we have
\[
\mathbb{E}\left[e^{\gamma  \min_{i=1, \ldots, d}r_i(\eta_1)}\right]<\infty ,\qquad\text{and}\qquad \mathbb{E}\Big[\max_{i=1,\ldots,d}r_i(\eta_1)\Big]<\infty .
\]
\item[(\namedlabel{A:support}{$\mathcal{A}4$})] (Full support). 
For any $\pi\in \mathcal{S}_0$, $t>0$, and any non-empty open set $C\subset \mathcal{S}$ we have $\mathbb{P}[G(\pi,R(t))\in C]>0$.
\end{enumerate}

Let us now provide some comments on these conditions. Assumption~\eqref{B:increments} ensures that the controlled process is Markovian between the interventions. It is satisfied e.g. for the exponentiated L\'evy processes.  Next, Assumption~\eqref{B:waiting_times} is a mild condition stating that there are no immediate investment opportunities and the expected waiting time is finite. It should be noted that we do not impose any additional distributional assumptions on the waiting times and, in particular, we can set $\eta_i\equiv 1$. Next, Assumption~\eqref{A:integrability_log} ensures the integrability of the suitable log-returns. 
 Finally, Assumption~\eqref{A:support} reflects the non-degeneracy of the problem, where the (uncontrolled) weight process could explore the whole state space. It is satisfied e.g. for $R(t)$ following a multivariate log-normal distribution. It should be noted that Assumption~\eqref{A:support} is used only in selected results, most notably Theorem~\ref{th:unique_maximiser}. In particular, the existence of a solution to the underlying Bellman equation can be proved without~\eqref{A:support}; see Theorem~\ref{th:existence_concave} for details.

Before we proceed, we show a technical result which facilitates a step-wise computation of certain expressions related to the optimality functional.

\begin{lemma}\label{lm:telescope}
    For any $\pi\in \mathcal{S}$ and $0\leq t\leq s\leq u<\infty$, we have
    \begin{align}
        \ln \langle \pi, R(t,u)\rangle = \ln \langle \pi, R(t,s)\rangle+\ln \langle G(\pi,R(t,s)),R(s,u)\rangle. 
    \end{align}
\end{lemma}
\begin{proof}
    By a direct computation, we have
    \begin{align*}
        \langle G(\pi,R(t,s)),R(s,u)\rangle = \left\langle \frac{\pi\cdot R(t,s)}{\langle\pi, R(t,s) \rangle} ,R(s,u)\right\rangle = \frac{\langle \pi, R(t,s)\cdot R(s,u)\rangle}{\langle\pi, R(t,s) \rangle}=\frac{\langle \pi, R(t,u)\rangle}{\langle\pi, R(t,s) \rangle}.
    \end{align*}
    Taking the logarithms and rearranging terms, we conclude the proof.
\end{proof}

To solve the problem linked to~\eqref{eq:J(pi)_log}, for any $\gamma \leq 0$, we will show the existence of a pair $\bar{v}^\gamma\in \mathcal{C}_b(\mathcal{S})$ and $\bar\lambda^\gamma\in \mathbb{R}$ such that the following Bellman equation is satisfied
\begin{equation}\label{eq:Bellman_log}
    \bar{v}^\gamma(\pi)=\sup_{\pi'\in \mathcal{S}}\left( \ln s(\pi,\pi')-\bar\lambda^\gamma\mathbb{E}[\eta_1]+\mu^\gamma\left[  \ln\langle\pi',R(\eta_1)\rangle+\bar{v}^\gamma (G(\pi',R(\eta_1)))\right]\right), \quad \pi\in \mathcal{S}.
\end{equation}
Note that, for $\gamma<0$, this equation could be equivalently expressed as
\begin{align}\label{eq:Bellman_exponentiated}
    v^\gamma(\pi)&=\inf_{\pi'\in \mathcal{S}}\left( \gamma\ln s(\pi,\pi')-\lambda^\gamma\mathbb{E}[\eta_1] + \ln \mathbb{E}\left[e^{\gamma\ln\langle\pi',R(\eta_1)\rangle+v^\gamma (G(\pi',R(\eta_1)))}\right]\right), \quad  \pi\in \mathcal{S},
\end{align}
where we set $v^\gamma:=\gamma \bar{v}^\gamma$ and $\lambda^\gamma:=\gamma \bar{\lambda}^\gamma$. For $\gamma=0$, Equation~\eqref{eq:Bellman_log} reads
\begin{align}\label{eq:Bellman_RN}
    \bar{v}^0(\pi)&=\sup_{\pi'\in \mathcal{S}}\left( \ln s(\pi,\pi')-\bar\lambda^0\mathbb{E}[\eta_1] + \mathbb{E}\left[\ln\langle\pi',R(\eta_1)\rangle+\bar{v}^0 (G(\pi',R(\eta_1)))\right]\right), \quad \pi\in \mathcal{S}.
\end{align}

Now, we link a solution to~\eqref{eq:Bellman_log} with the optimal values of the underlying control problems. Also, we show that the strategies given by the maximisers of the right-hand side of the associated Bellman equations are optimal.

\begin{theorem}[Verification theorem]\label{th:verification}
\noindent
Assume~\eqref{B:increments},~\eqref{B:waiting_times},  and~\eqref{A:integrability_log}. Also, let $\gamma\leq 0$ and let $(\bar{v}^\gamma, \bar\lambda^\gamma)\in \mathcal{C}_b(\mathcal{S})\times \mathbb{R}$ be a solution to~\eqref{eq:Bellman_log}. Next, let $\hat{\mathbf\Pi}:=(\hat\pi(t))_{t\geq 0}$ be a strategy given, for any $n\in \mathbb{N}$, by
\begin{align}\label{eq:th:verification:strategy}
    \hat\pi(\tau_n)&:=\argmax_{\pi'\in \mathcal{S}}\left( \ln s(\hat\pi(\tau_n^-),\pi')-\bar\lambda^\gamma\mathbb{E}[\eta_1]+\mu^\gamma\left[   \ln\langle\pi',R(\eta_1)\rangle+\bar{v}^\gamma (G(\pi',R(\eta_1)))\right]\right),\nonumber\\
    \hat\pi(t)&:=G(\hat\pi(\tau_n),R(\tau_n,t)), \quad t\in [\tau_n, \tau_{n+1}).
\end{align}
Then, we get
\[
\sup_{\mathbf\Pi\in \mathcal{A}(\pi)}J^\gamma(\mathbf\Pi,\pi) = J^\gamma(\hat{\mathbf\Pi},\pi)=\bar\lambda^\gamma, \quad \pi\in \mathcal{S}.
\]
\end{theorem}
\begin{proof}
    The argument is standard, so we provide only an outline. Iterating~\eqref{eq:Bellman_log} and using Assumption~\eqref{B:increments} and~\eqref{B:waiting_times} combined with Lemma~\ref{lm:telescope}, for any $\pi\in \mathcal{S}$, $\mathbf\Pi \in \mathcal{A}(\pi)$, and $n\in \mathbb{N}$, $n\geq 1$, we get
    \begin{multline*}
        \bar{v}^\gamma(\pi)+n\bar\lambda^\gamma \mathbb{E}[\eta_1]\\
        \geq \mu^\gamma\bigg[\sum_{i=0}^{n-1}\left( \ln s(\pi(\tau_i^-),\pi(\tau_i))+\ln \langle\pi(\tau_i),R(\tau_i,\tau_{i+1})\rangle\right)\\
        +\bar{v}^\gamma(G(\pi(\tau_{n-1}), R(\tau_{n-1}, \tau_n)))  \bigg].
    \end{multline*}
    Thus, dividing both sides by $n \mathbb{E}[\eta_1]$, using the fact that $\bar{v}^\gamma$ is bounded and $\mathbb{E}[\tau_n]=n\mathbb{E}[\eta_1]$, and letting $n\to\infty$, we obtain
    \[
    \bar\lambda^\gamma\geq J^\gamma(\mathbf\Pi,\pi).
    \]
    Using a similar argument applied to the maximiser strategy given by~\eqref{eq:th:verification:strategy}, we obtain
    \[
    \bar\lambda^\gamma= J^\gamma(\hat{\mathbf\Pi},\pi),
    \]
    which concludes the proof.
\end{proof}

\section{Solution to the Bellman equation}\label{S:Bellman}

In this section, we show the existence of a solution to the Bellman equation~\eqref{eq:Bellman_log} and study the
maximisers on its right-hand side. The argument is based on a fixed-point theorem of Schauder type applied to a suitable
compact convex family of functions on $\mathcal{S}$. In contrast to a Krein--Rutman-type approach or a span-contraction argument, no auxiliary
diversification constraint is needed. Also, the solution obtained is shown to have a concave exponential lift, which is the key
structural property used in Section~\ref{S:unique_maximiser_subsec} to prove that the optimal rebalancing is unique.

Throughout this section we assume~\eqref{B:increments},~\eqref{B:waiting_times}, and~\eqref{A:integrability_log}; the
full-support condition~\eqref{A:support} is used \emph{only} in Section~\ref{S:unique_maximiser_subsec} and is invoked
explicitly there. In particular, neither the existence of a solution nor any of the bounds below depend
on~\eqref{A:support}.


\subsection{Capital coordinates and the self-financing cone}\label{S:cone}

Recalling~\eqref{eq:proportional_cost}, we note that $p$ is convex, positively homogeneous, non-negative, and satisfies
\begin{equation}\label{eq:p_Lipschitz}
    |p(x)-p(x')|\leq \bar{c}\,\Vert x-x'\Vert_1, \quad x,x'\in \mathbb{R}^d,
\end{equation}
where $\bar{c}:=\max_{i=1,\ldots,d}\max(c_i,h_i)\in [0,1)$. 
Let us recall the map $s$ introduced in Section 2 to measure the capital reduction stemming from portfolio change, see Equation~\eqref{eq:wealth_ratio_transactions} for details. Now, we collect the properties of $s$ used in the sequel. In particular, we amend Lemma 2.1 in~\cite{PitSte2023} by a semi-explicit representation of $s$; cf. Equation~\eqref{eq:sclosed}.
Throughout, we set
\[
\bar{s}:=\frac{1-\bar{c}}{1+\bar{c}}\in (0,1].
\]

\begin{lemma}\label{lm:s_properties}
Let $\Bbox:=\prod_{i=1}^{d}[-h_i,c_i]$. The following statements hold:
\begin{enumerate}[label=(\roman*)]
    \item For all $\pi,\pi'\in\mathcal{S}$, we have
    \begin{equation}\label{eq:sclosed}
    s(\pi,\pi')=\min_{a\in\Bbox}\ \frac{1+\langle a,\pi\rangle}{1+\langle a,\pi'\rangle}
    =\min_{a\in\prod_{i=1}^d\{-h_i,c_i\}}\ \frac{1+\langle a,\pi\rangle}{1+\langle a,\pi'\rangle}.
    \end{equation}
    \item For any $\pi,\pi'\in \mathcal{S}$ we have $s(\pi,\pi')\in [\bar{s},1]$.
    \item For any $\pi_1,\pi_2,\pi_1', \pi_2'\in \mathcal{S}$ we have 
    \[
    |s(\pi_1,\pi_1')-s(\pi_2,\pi_2')|\leq
    \frac{\bar{c}}{1-\bar{c}}\left(\Vert \pi_1-\pi_2\Vert_1+\Vert \pi_1'-\pi_2'\Vert_1 \right).
    \]
\end{enumerate}
\end{lemma}
\begin{proof}

(i) Since $\Bbox$ is a box and the maximum over an
interval of a linear function is attained at an endpoint, for any $x\in \mathbb{R}^d$, we have 
\begin{equation}\label{eq:lm:s_properties:1}
    \max_{a\in\Bbox}\langle a,x\rangle=\sum_{i=1}^d\max(c_ix_i,-h_ix_i)=p(x).
\end{equation}
Next, recall that by Lemma 2.1 in~\cite{PitSte2023}, the value $s(\pi,\pi')$ is given as a unique solution to 
\begin{equation}\label{eq:sdef}
    F(s(\pi,\pi'),\pi,\pi')=1
\end{equation}
where $F(w,\pi,\pi') = w+p(w\pi'-\pi)$, $w>0$, $\pi,\pi'\in \mathcal{S}$. Combining this with~\eqref{eq:lm:s_properties:1}, for $w>0$, $\pi,\pi'\in \mathcal{S}$, we obtain
\[
F(w,\pi,\pi')=w+\max_{a\in\Bbox}\big\langle a,w\pi'-\pi\big\rangle
=\max_{a\in\Bbox}\Big\{w\big(1+\langle a,\pi'\rangle\big)-\langle a,\pi\rangle\Big\}.
\]
As $|\langle a,\pi\rangle|,|\langle a,\pi'\rangle|\le\bar c$, the quantity $1+\langle a,\pi'\rangle$ lies in the interval
$[1-\bar c,1+\bar c]\subset(0,\infty)$ for any $a \in \Bbox$, so
\[
F(w,\pi,\pi')\le1\iff w\le\frac{1+\langle a,\pi\rangle}{1+\langle a,\pi'\rangle}\ \ \forall a\in\Bbox .
\]
Also $w\mapsto F(w,\pi,\pi')$ is continuous, strictly increasing,
with $F(0,\pi,\pi')=\langle h,\pi\rangle<1$; hence \eqref{eq:sdef} has a unique root, equal to the middle expression in
\eqref{eq:sclosed}. The vertex form follows because a linear-fractional function attains its extrema over a polytope at a
vertex.

(ii) The numerators in~\eqref{eq:sclosed} are bounded from below by $1-\bar c$, the denominators are bounded from above by $1+\bar c$, so $s(\pi,\pi')\ge\bar s$; taking $a=0\in\Bbox$ gives $s(\pi,\pi')\le1$.

(iii) Write $w_k:=s(\pi_k,\pi'_k)\in (0,1]$, $k=1,2$, and recall from~\eqref{eq:sdef} that $w_k+p(w_k\pi_k'-\pi_k)=1$. Subtracting
these identities and using~\eqref{eq:p_Lipschitz} as well as $\Vert \pi_1'\Vert_1=1$, we get
\begin{align*}
    |w_1-w_2|&=|p(w_2\pi_2'-\pi_2)-p(w_1\pi_1'-\pi_1)|\\
    &\leq \bar{c}\Vert w_2\pi_2'-w_1\pi_1'-(\pi_2-\pi_1)\Vert_1\\
    & = \bar{c}\Vert w_2(\pi_2'-\pi_1')+(w_2-w_1)\pi_1'-(\pi_2-\pi_1)\Vert_1\\
    &\leq
\bar{c}\Vert\pi_2'-\pi_1'\Vert_1 +\bar{c}|w_2-w_1| +\bar{c}\Vert \pi_2-\pi_1\Vert_1,
\end{align*}
and the claim follows by rearranging the terms.
\end{proof}

In the following, we show that the maximiser in~\eqref{eq:Bellman_log} is unique. A natural approach would build on the concavity of the map under the supremum sign; in particular, we would need to show that the map $s$ is concave. However, as we demonstrate in the following example, this is not true.

\begin{example}\label{ex:non_concave}

    Take $d=2$ and $c=h=(0.5,0.5)$, so $p(z)=0.5(|z_1|+|z_2|)$. Parametrising $\pi=(u,1-u)$,
$\pi'=(t,1-t)$, $u,t\in [0,1]$, we conclude that~\eqref{eq:sclosed} collapses to
\[
s((u,1-u),(t,1-t))=\min\left(\frac{u+0.5}{t+0.5},\ \frac{1.5-u}{1.5-t}\right).
\]
In particular, for $\pi=(\tfrac12,\tfrac12)$ and $t\ge\tfrac12$, the first branch is active and
$s\big(\pi,(t,1-t)\big)=\tfrac{2}{2t+1}$, which is \emph{strictly convex} in $t$. Concretely, with
$\pi'^{1}=(\tfrac35,\tfrac25)$, $\pi'^{2}=(1,0)$ and their midpoint $\pi'^{m}=(\tfrac45,\tfrac15)$, we have $s(\pi,\pi'^{1})=\tfrac{10}{11}$, $s(\pi,\pi'^{m})=\tfrac{10}{13}$, $s(\pi,\pi'^{2})=\tfrac23$.
Hence
\[
\tfrac12 s(\pi,\pi'^{1})+\tfrac12 s(\pi,\pi'^{2})=\tfrac{26}{33}\approx0.7879\ >\ \tfrac{10}{13}\approx0.7692=s(\pi,\pi'^{m}),
\]
so concavity fails.
\end{example}

Recall that in~\eqref{eq:J(pi)_log} we introduced the weights-based representation of the problem. In the following, to overcome the issue with the lack of concavity stated in Example~\ref{ex:non_concave}, we also use capital-based representation. More specifically, for $x,y\in \mathbb{R}^d_+$ we define the \emph{cost-to-reach} functional
\begin{equation}\label{eq:Theta}
    \Theta(x,y):=\langle y,\mathbf{1}\rangle+p(y-x),
\end{equation}
so that, in view of~\eqref{eq:self_financing_fp}, the rebalancing from $x$ to $y$ is self-financing if and only if
$\Theta(x,y)=\langle x,\mathbf{1}\rangle$. Accordingly, we introduce the \emph{self-financing frontier}, the \emph{budget set}, and the
\emph{budget cone} given by
\begin{align}\label{eq:budget_sets}
    \Sigma(x)&:=\{y\in \mathbb{R}^d_+\colon \Theta(x,y)=\langle x,\mathbf{1}\rangle\}, \quad
    \mathcal{K}(x):=\{y\in \mathbb{R}^d_+\colon \Theta(x,y)\leq \langle x,\mathbf{1}\rangle\}, \nonumber\\
    \Gamma&:=\{(x,y)\colon x \in \mathbb{R}^d_+,\, y \in \mathcal{K}(x)\},
\end{align}
respectively. The set $\mathcal{K}(x)$ is the \emph{free-disposal relaxation} of $\Sigma(x)$, in which the investor is
allowed to discard capital; it will be shown in Lemma~\ref{lm:reward}(ii) that this relaxation is inactive at the optimum.

Let us relate this notation to the map $s$ analysed in Lemma~\ref{lm:s_properties}. For $\pi,\pi'\in \mathcal{S}$, the
self-financing condition~\eqref{eq:self_financing_weights2} with $w:=W(t)/W(t-)$ reads $w+p(w\pi'-\pi)=1$, that is,
\begin{equation}\label{eq:s_as_root}
    \Theta(\pi, s(\pi,\pi')\pi')=1=\langle \pi,\mathbf{1}\rangle,
\end{equation}
so that $s(\pi,\pi')$ is exactly the unique scalar $w>0$ for which $w\pi'$ lies on the frontier $\Sigma(\pi)$.

Now, we discuss the properties of budget sets.

\begin{lemma}\label{lm:cone}
The following statements hold:
\begin{enumerate}[label=(\roman*)]
    \item The map $\Theta$ is jointly convex and jointly positively homogeneous on $\mathbb{R}^d_+\times \mathbb{R}^d_+$.
    In particular, $\Gamma$ is a closed convex cone, the set $\mathcal{K}(x)$ is convex and compact,
    $\mathcal{K}(\alpha x)=\alpha \mathcal{K}(x)$ for $\alpha>0$, and $y^k\in \mathcal{K}(x^k)$, $k=1,2$, implies
    $ty^1+(1-t)y^2\in \mathcal{K}(tx^1+(1-t)x^2)$ for $t\in [0,1]$.
    \item For any $x,y\in \mathbb{K}$ there exists exactly one $\alpha(x,y)>0$ such that $\alpha(x,y)\,y\in \Sigma(x)$.
    Moreover, the map $\beta\to \Theta(x,\beta y)$ is strictly increasing on $[0,\infty)$, and for any $\beta>0$,
    \[
    \beta y \in \mathcal{K}(x) \iff \beta \leq \alpha(x,y), \qquad \beta y \in \Sigma(x) \iff \beta = \alpha(x,y).
    \]
    \item We have $x\in \Sigma(x)$ and $\bar{s}\langle x,\mathbf{1}\rangle\leq \langle y,\mathbf{1}\rangle\leq \langle
    x,\mathbf{1}\rangle$ for any $y\in \Sigma(x)$.
    \item For any $\pi \in \mathcal{S}$, the map $\iota_\pi\colon \mathcal{S}\to \Sigma(\pi)$ given by
    $\iota_\pi(\pi'):=s(\pi,\pi')\pi'$ is a homeomorphism of $\mathcal{S}$ onto $\Sigma(\pi)$, with the inverse given by
    $\Sigma(\pi)\ni y\to y/\langle y,\mathbf{1}\rangle\in \mathcal{S}$.
\end{enumerate}
\end{lemma}
\begin{proof}
(i) On $\mathbb{R}^d_+\times \mathbb{R}^d_+$ we may write
\[
\Theta(x,y)-\langle x,\mathbf{1}\rangle=\langle y,\mathbf{1}\rangle-\langle x,\mathbf{1}\rangle+p(y-x),
\]
which is a linear function of $(x,y)$ plus a convex function precomposed with a linear map, hence convex; it is also
positively homogeneous. Thus $\Gamma$ is the $0$-sublevel set of a convex positively homogeneous function on the convex cone
$\mathbb{R}^d_+\times \mathbb{R}^d_+$, i.e.\ a closed convex cone. Compactness of $\mathcal{K}(x)$ follows from $p\geq 0$,
which forces $\langle y,\mathbf{1}\rangle\leq \langle x,\mathbf{1}\rangle$ for $y \in \mathcal{K}(x)$. The remaining properties follow from homogeneity and compactness of $\Theta$ by a straightforward verification.

(ii) Let $g(\beta):=\Theta(x,\beta y)=\beta\langle y,\mathbf{1}\rangle+p(\beta y-x)$ for $\beta\geq 0$. For
$0\leq \beta_1<\beta_2$, by~\eqref{eq:p_Lipschitz} we get $p(\beta_2y-x)-p(\beta_1y-x)\geq
-\bar{c}(\beta_2-\beta_1)\langle y,\mathbf{1}\rangle$, hence
\[
g(\beta_2)-g(\beta_1)\geq (1-\bar{c})(\beta_2-\beta_1)\langle y,\mathbf{1}\rangle>0.
\]
Thus $g$ is (strictly) increasing and $g(\beta)\to\infty$ as $\beta\to\infty$. Since $g(0)=p(-x)=\langle
h,x\rangle\leq \bar{c}\langle x,\mathbf{1}\rangle<\langle x,\mathbf{1}\rangle$, the continuous strictly increasing map $g$
attains the level $\langle x,\mathbf{1}\rangle$ at exactly one point, which we denote by $\alpha(x,y)>0$. Also,
$g(\beta)\leq \langle x,\mathbf{1}\rangle$ if and only if $\beta\leq \alpha(x,y)$, which gives the stated characterisation.

(iii) We have $\Theta(x,x)=\langle x,\mathbf{1}\rangle$ hence $x\in \Sigma(x)$. Next, for $y\in \Sigma(x)$, from $p\geq 0$ we get $\langle
y,\mathbf{1}\rangle\leq \langle x,\mathbf{1}\rangle$, while~\eqref{eq:p_Lipschitz} and the triangle inequality
give $p(y-x)\leq \bar{c}\Vert y-x\Vert_1\leq \bar{c}(\langle y,\mathbf{1}\rangle+\langle x,\mathbf{1}\rangle)$. Thus
$\langle x,\mathbf{1}\rangle\leq \langle y,\mathbf{1}\rangle+\bar{c}(\langle y,\mathbf{1}\rangle+\langle
x,\mathbf{1}\rangle)$, i.e.\ $\langle y,\mathbf{1}\rangle\geq \bar{s}\langle x,\mathbf{1}\rangle$.

(iv) If $y=\iota_\pi(\pi')$, then $y\in \Sigma(\pi)$ by~\eqref{eq:s_as_root}, and $\langle
y,\mathbf{1}\rangle=s(\pi,\pi')>0$, so that $y/\langle y,\mathbf{1}\rangle=\pi'$. Conversely, if $y\in \Sigma(\pi)$, then
$y\neq \mathbf{0}$, since $\Theta(\pi,\mathbf{0})=\langle h,\pi\rangle<1 = \langle \pi, \mathbf{1}\rangle$; thus, setting $\sigma:=\langle
y,\mathbf{1}\rangle>0$ and $\pi':=y/\sigma\in \mathcal{S}$, we get $\Theta(\pi,\sigma\pi')=1$, so that
$\sigma=s(\pi,\pi')$ by~\eqref{eq:s_as_root} and (ii), i.e.\ $y=\iota_\pi(\pi')$. The continuity of
$\iota_\pi$ and of its inverse follows from Lemma~\ref{lm:s_properties} (i) and (ii).
\end{proof}

\subsection{The entropic utility and the reward functional}\label{S:reward}

We now record the properties of the entropic utility $\mu^\gamma$ used in the sequel: for any $\gamma\leq 0$, $t\in [0,1]$, and sufficiently integrable $Z,Z_1,Z_2$, and $a\in \mathbb{R}$, we have
\begin{align}\label{eq:mu_basic}
    Z_1\leq Z_2 \quad \mathbb{P}\, a.s.&\implies \mu^\gamma[Z_1]\leq \mu^\gamma[Z_2], \nonumber \\ 
    \mu^\gamma[Z+a]&=\mu^\gamma[Z]+a, \nonumber\\
    |\mu^\gamma[Z_1]-\mu^\gamma[Z_2]|&\leq \Vert Z_1-Z_2\Vert,\nonumber\\\
    \mu^\gamma[tZ_1+(1-t)Z_2]&\geq t \mu^\gamma[Z_1] + (1-t) \mu^\gamma[Z_2].
\end{align}
Note that these properties state monotonicity, cash-additivity, Lipschitz property, and concavity of the entropic utility, respectively.
The first three properties follow directly from the definition. For the fourth one, use the H\"{o}lder inequality or Example 4.13 in~\cite{FolSch2016}. Also, note that directly from the properties of expectation we get that if $Z_1\leq Z_2$ a.s. and $\mu^\gamma[Z_1]=\mu^\gamma[Z_2]\in \mathbb{R}$, then $Z_1=Z_2$ a.s.

In the following, we use also an exponentiated version of $\mu^\gamma$.

\begin{lemma}\label{lm:mu_concave}
Let $\gamma\leq 0$. Then, on the family of random variables $Z$ such that $\mu^\gamma[Z]<\infty$, the map $Z\to
    e^{\mu^\gamma[\ln Z]}$ is positively homogeneous and superadditive, i.e.
\begin{equation}\label{eq:lm:mu_concave:1}
        e^{\mu^\gamma[\ln (Z_1+Z_2)]}\geq e^{\mu^\gamma[\ln Z_1]}+e^{\mu^\gamma[\ln Z_2]};
\end{equation}
    in particular, it is concave.
\end{lemma}
\begin{proof}
For $\gamma<0$ we have $e^{\mu^\gamma[\ln Z]}=(\mathbb{E}[Z^\gamma])^{1/\gamma}$, so the positive homogeneity is clear.
For the superadditivity, set $S:=Z_1+Z_2$ and note that, with the conjugate exponents $\gamma<0$ and
$q:=\gamma/(\gamma-1)\in (0,1)$, the reverse H\"older inequality gives
\[
\mathbb{E}[Z_kS^{\gamma-1}]\geq (\mathbb{E}[Z_k^\gamma])^{1/\gamma}\,(\mathbb{E}[S^{\gamma}])^{(\gamma-1)/\gamma}, \quad k=1,2.
\]
Summing over $k=1,2$ and using $Z_1S^{\gamma-1}+Z_2S^{\gamma-1}=S^\gamma$, we obtain
$(\mathbb{E}[S^\gamma])^{1/\gamma}\geq (\mathbb{E}[Z_1^\gamma])^{1/\gamma}+(\mathbb{E}[Z_2^\gamma])^{1/\gamma}$, which is the
claim. Superadditivity combined with the positive homogeneity gives the concavity.

For $\gamma=0$, we have $e^{\mu^\gamma[\ln Z]}=e^{\mathbb{E}[\ln Z]}$, so the claim is the integral version of Mahler's inequality (see e.g.~\cite[Theorem 11]{HarLitPol1934}). Nevertheless, for completeness, we provide a proof. Let \[
  a:=e^{\mathbb{E}[\ln Z_1]}\in(0,\infty),\qquad
  b:=e^{\mathbb{E}[\ln Z_2]}\in(0,\infty),\qquad
  \lambda:=\frac{a}{a+b}\in(0,1).
\]
Then $1-\lambda=b/(a+b)$, and hence $\lambda/a=(1-\lambda)/b=1/(a+b)$, which yields
the pointwise identity
\begin{equation}\label{eq:convexcomb}
  \frac{Z_1+Z_2}{a+b}
  \;=\;\lambda\,\frac{Z_1}{a}\;+\;(1-\lambda)\,\frac{Z_2}{b}
  \qquad\text{on }\Omega .
\end{equation}
Applying to the right-hand side of \eqref{eq:convexcomb} the weighted
arithmetic--geometric mean inequality
$\lambda x+(1-\lambda)y\geq x^{\lambda}y^{1-\lambda}$, $x,y>0$, we get
\begin{equation}\label{eq:amgm}
  \frac{Z_1+Z_2}{a+b}\;\geq\;
  \left(\frac{Z_1}{a}\right)^{\lambda}\left(\frac{Z_2}{b}\right)^{1-\lambda}
  \qquad\mathbb{P}\text{-a.s.}
\end{equation}
Taking logarithms in \eqref{eq:amgm} and then expectations of both sides, we obtain
\[
  \mathbb{E}\big[\ln(Z_1+Z_2)\big]-\ln(a+b)
  \;\geq\;
  \lambda\big(\mathbb{E}[\ln Z_1]-\ln a\big)
  +(1-\lambda)\big(\mathbb{E}[\ln Z_2]-\ln b\big)
  \;=\;0,
\]
where the last equality follows from $\mathbb{E}[\ln Z_1]=\ln a$ and
$\mathbb{E}[\ln Z_2]=\ln b$. Consequently
$\mathbb{E}[\ln(Z_1+Z_2)]\geq\ln(a+b)$, and exponentiating gives~\eqref{eq:lm:mu_concave:1} for $\gamma=0$.
\end{proof}

Now, we state a technical result which is a version of the Feller property for the underlying controlled process; it also
shows that the right-hand side of~\eqref{eq:Bellman_log} is well defined.

\begin{lemma}\label{lm:Feller}
Let $\gamma\leq 0$ and assume~\eqref{A:integrability_log}. Then, for any $v\in \mathcal{C}_b(\mathcal{S})$, the map
\begin{equation}\label{eq:lm:Feller:1}
    \mathcal{S}\ni \pi'\to \Psi_v(\pi'):= \mu^\gamma\left[\ln \langle \pi',R(\eta_1)\rangle+v(G(\pi',R(\eta_1)))\right]\in \mathbb{R}
\end{equation}
is well defined, bounded, and continuous.
\end{lemma}
\begin{proof}
Fix $\pi'\in \mathcal{S}$. Since $R_i(\eta_1)=e^{r_i(\eta_1)}$ and $\langle \pi', \mathbf{1}\rangle =1$, we get
\[
e^{\min_{i=1,\ldots,d}r_i(\eta_1)} = \langle \pi',e^{\min_{i=1,\ldots,d}r_i(\eta_1)}\rangle\leq  \langle \pi',R(\eta_1)\rangle\leq
e^{\max_{i=1,\ldots,d}r_i(\eta_1)},
\]
so that, writing $Z_{\pi'}:=\ln \langle \pi',R(\eta_1)\rangle+v(G(\pi',R(\eta_1)))$, we obtain
\begin{equation}\label{eq:lm:Feller:2}
    \min_{i=1,\ldots,d}r_i(\eta_1)-\Vert v\Vert\ \leq\ Z_{\pi'}\ \leq\ \max_{i=1,\ldots,d}r_i(\eta_1)+\Vert v\Vert.
\end{equation}
For $\gamma<0$, we get $e^{\gamma Z_{\pi'}}\leq e^{|\gamma|\Vert v\Vert}e^{\gamma
\min_{i=1,\ldots,d}r_i(\eta_1)}$, which is integrable by Assumption~\eqref{A:integrability_log}; consequently
$0<\mathbb{E}[e^{\gamma Z_{\pi'}}]<\infty$ and $\Psi_v(\pi')=\frac{1}{\gamma}\ln \mathbb{E}[e^{\gamma Z_{\pi'}}]$ is a finite
real number. For $\gamma=0$, Assumption~\eqref{A:integrability_log} gives $\mathbb{E}|Z_{\pi'}|<\infty$, so
that $\Psi_v(\pi')=\mathbb{E}[Z_{\pi'}]$ is finite. In both cases, the envelopes in~\eqref{eq:lm:Feller:2} do not depend on
$\pi'$, which, by the dominated convergence theorem combined with the
continuity of $\pi'\to Z_{\pi'}$, shows the continuity of $\Psi_v$.
\end{proof}

Next, for $v\in \mathcal{C}_b(\mathcal{S})$ we define its \emph{positively homogeneous lift} $\mathcal{W}_v$ by
\begin{equation}\label{eq:lift}
    \mathcal{W}_v(x):=\langle x,\mathbf{1}\rangle\, e^{v(x/\langle x,\mathbf{1}\rangle)}, \quad x\in \mathbb{K}, \qquad
    \mathcal{W}_v(\mathbf{0}):=0.
\end{equation}
Note that $\mathcal{W}_v$ is continuous on $\mathbb{R}^d_+$, positively homogeneous, restricts to $e^{v}$ on $\mathcal{S}$,
and satisfies
\begin{equation}\label{eq:lift_bounds}
    e^{-\Vert v\Vert}\langle x,\mathbf{1}\rangle\leq \mathcal{W}_v(x)\leq e^{\Vert v\Vert}\langle x,\mathbf{1}\rangle, \quad x \in \mathbb{R}^d_+.
\end{equation}
The quantity $\mathcal{W}_v(x)$ should be understood as the wealth-scaled exponential of the value function; it is not to be
confused with the wealth process $W$ from~\eqref{eq:wealth_process}. Finally, for a continuous positively homogeneous
$\mathcal{W}\colon \mathbb{R}^d_+\to [0,\infty)$ which is positive on $\mathbb{K}$, we define the \emph{one-period reward}
\begin{equation}\label{eq:reward}
    \Phi_{\mathcal{W}}(y):=e^{\mu^\gamma\left[\ln \mathcal{W}(y\cdot R(\eta_1))\right]}, \quad y \in \mathbb{K}, \quad \Phi_{\mathcal{W}}(\mathbf{0})=0,
\end{equation}
i.e.\ the exponentiated post-growth value of the post-transaction position $y$. Also, note that with the map $\Psi_v$ from Lemma~\ref{lm:Feller}, we have $\Psi_v=\ln \Phi_{\mathcal{W}_v}$, $v\in \mathcal{C}_b(\mathcal{S})$.

Now, we provide several properties of the map $\Phi_{\mathcal{W}}$.

\begin{lemma}\label{lm:reward}
Let $\gamma\leq 0$ and assume~\eqref{A:integrability_log}. Also, let $\mathcal{W}\colon \mathbb{R}^d_+\to [0,\infty)$ be continuous, positively homogeneous, positive on $\mathbb{K}$, and
such that $c_1\langle y,\mathbf{1}\rangle\leq \mathcal{W}(y)\leq c_2\langle y,\mathbf{1}\rangle$, $y\in \mathbb{R}^d_+$, for
some $0<c_1\leq c_2<\infty$. Then, the following statements hold:
\begin{enumerate}[label=(\roman*)]
    \item The map $\Phi_{\mathcal{W}}$ is finite, continuous, and positively homogeneous on $\mathbb{R}^d_+$, with
    $\Phi_{\mathcal{W}}(y)>0$ for $y\neq \mathbf{0}$.
    \item For any $x\in \mathbb{K}$, the supremum $\sup_{y\in \mathcal{K}(x)}\Phi_{\mathcal{W}}(y)$ is attained, it is equal
    to $\sup\limits_{y\in \Sigma(x)}\Phi_{\mathcal{W}}(y)$, and every maximiser over $\mathcal{K}(x)$ belongs to $\Sigma(x)$.
    \item If $\mathcal{W}=\mathcal{W}_v$ for some $v \in \mathcal{C}_b(\mathcal{S})$, then, for any $\pi,\pi'\in
    \mathcal{S}$,
    \begin{equation}\label{eq:objective_match}
        \ln \Phi_{\mathcal{W}_v}(\iota_\pi(\pi'))=\ln s(\pi,\pi')+\Psi_v(\pi').
    \end{equation}
    \item If, in addition, the map $\mathcal{W}$ is concave on $\mathbb{R}^d_+$, then so is $\Phi_{\mathcal{W}}$.
\end{enumerate}
\end{lemma}
\begin{proof}
(i) The positive homogeneity follows from the positive homogeneity of $\mathcal{W}$ and the cash-additivity
in~\eqref{eq:mu_basic}. Next, from the assumed bounds on
$\mathcal{W}$ and the monotonicity in~\eqref{eq:mu_basic}, we get
\begin{equation}\label{eq:lm:reward:1}
    c_1e^{\mu^\gamma[\ln \langle y,R(\eta_1)\rangle]}\leq \Phi_{\mathcal{W}}(y)\leq c_2e^{\mu^\gamma[\ln \langle
    y,R(\eta_1)\rangle]}, \quad y \in \mathbb{R}^d_+.
\end{equation}
For $y\in \mathbb{K}$, applying Lemma~\ref{lm:Feller} with $v\equiv 0$ to the weight vector $y/\langle
y,\mathbf{1}\rangle$ and using the cash-additivity, we get that $\mu^\gamma[\ln \langle y,R(\eta_1)\rangle]$ is finite and
continuous in $y$; together with~\eqref{eq:lm:reward:1} this gives the finiteness, positivity, and continuity of
$\Phi_{\mathcal{W}}$ on $\mathbb{K}$. Finally, $\Phi_{\mathcal{W}}(y)\leq c_2\langle
y,\mathbf{1}\rangle e^{\mu^\gamma[\max_{i=1,\ldots,d}r_i(\eta_1)]}\to 0$ as $y\to \mathbf{0}$, which shows the continuity at
$\mathbf{0}$; here the exponent is finite, since $\mu^\gamma[Z]\leq \mathbb{E}[Z]$ for $\gamma \leq 0$ and
$\mathbb{E}[\max_{i=1,\ldots,d}r_i(\eta_1)]<\infty$ by Assumption~\eqref{A:integrability_log}.

(ii) The set $\mathcal{K}(x)$ is compact by Lemma~\ref{lm:cone}(i) and $\Phi_{\mathcal{W}}$ is continuous, so the supremum is
attained. Let $y\in \mathcal{K}(x)\setminus \Sigma(x)$ with $y\neq \mathbf{0}$ and let $\alpha:=\alpha(x,y)$ be as in
Lemma~\ref{lm:cone}(ii). By the characterisation stated there, a multiple $\beta y$ belongs to $\mathcal{K}(x)$ if and only
if $\beta\leq \alpha$; applying this with $\beta=1$ and using $y\in \mathcal{K}(x)$, we get $\alpha\geq 1$, while $y\notin
\Sigma(x)$ excludes $\alpha=1$. Thus $\alpha>1$ and, by the positive homogeneity and positivity of $\Phi_{\mathcal{W}}$,
\[
\Phi_{\mathcal{W}}(\alpha y)=\alpha \Phi_{\mathcal{W}}(y)>\Phi_{\mathcal{W}}(y).
\]
Also, by definition of $\alpha$, we have $\alpha y \in \Sigma(x)\subset \mathcal{K}(x)$,
so that $y$ is not a maximiser. Finally, $y=\mathbf{0}$ is dominated by any point of $\Sigma(x)$. This shows that all
maximisers over $\mathcal{K}(x)$ belong to $\Sigma(x)$ and that the two suprema coincide.

(iii) Setting $y:=\iota_\pi(\pi')=s(\pi,\pi')\pi'$, we get $\langle y,R(\eta_1)\rangle=s(\pi,\pi')\langle
\pi',R(\eta_1)\rangle$ and $G(y,R(\eta_1))=G(\pi',R(\eta_1))$, so that
\[
\ln \mathcal{W}_v(y\cdot R(\eta_1))=\ln s(\pi,\pi')+\ln \langle \pi',R(\eta_1)\rangle+v(G(\pi',R(\eta_1))).
\]
Applying $\mu^\gamma$ and using the cash-additivity in~\eqref{eq:mu_basic} to pull out the constant $\ln s(\pi,\pi')$, we
get~\eqref{eq:objective_match}.

(iv) Let $y^1,y^2\in \mathbb{R}^d_+$, $t\in [0,1]$, and $y^t:=ty^1+(1-t)y^2$. Since the map $y\to y\cdot R(\eta_1)$ is
linear, the concavity of $\mathcal{W}$ gives, pointwise,
\[
\mathcal{W}(y^t\cdot R(\eta_1))\geq t\,\mathcal{W}(y^1\cdot R(\eta_1))+(1-t)\,\mathcal{W}(y^2\cdot R(\eta_1)).
\]
The map $Z\to e^{\mu^\gamma[\ln Z]}$ is monotone by~\eqref{eq:mu_basic} and concave by
Lemma~\ref{lm:mu_concave}. Applying it to both sides of the above inequality, we get
$\Phi_{\mathcal{W}}(y^t)\geq t\,\Phi_{\mathcal{W}}(y^1)+(1-t)\,\Phi_{\mathcal{W}}(y^2)$,
which concludes the proof.
\end{proof}

\subsection{Existence of a solution with a concave lift}\label{S:existence}

For $v \in \mathcal{C}_b(\mathcal{S})$ we define the Bellman operator $\mathcal{T}$ by
\begin{equation}\label{eq:T_operator}
    \mathcal{T}v(\pi):=\sup_{\pi'\in \mathcal{S}}\left(\ln s(\pi,\pi')+\Psi_v(\pi')\right), \quad \pi\in \mathcal{S},
\end{equation}
where $\Psi_v$ is given by~\eqref{eq:lm:Feller:1}. With this notation, a pair $(\bar{v}^\gamma,\bar\lambda^\gamma)\in
\mathcal{C}_b(\mathcal{S})\times \mathbb{R}$ is a solution to~\eqref{eq:Bellman_log} if and only if
\begin{equation}\label{eq:fixed_point}
    \mathcal{T}\bar{v}^\gamma=\bar{v}^\gamma+\bar\lambda^\gamma\mathbb{E}[\eta_1].
\end{equation}
Moreover, combining Lemma~\ref{lm:cone}(iv) with Lemma~\ref{lm:reward}(ii),(iii), we get that the lift of
$\mathcal{T}v$ is the maximal reward over the budget set, i.e.\ for any $\pi \in \mathcal{S}$,
\begin{equation}\label{eq:lift_of_T}
    e^{\mathcal{T}v(\pi)}=\sup_{y\in \Sigma(\pi)}\Phi_{\mathcal{W}_v}(y)=\sup_{y\in \mathcal{K}(\pi)}\Phi_{\mathcal{W}_v}(y).
\end{equation}
Indeed, by Lemma~\ref{lm:cone}(iv) we may parametrise $\Sigma(\pi)$ by $\pi'\in \mathcal{S}$ through
$y=\iota_\pi(\pi')$, and then~\eqref{eq:objective_match} identifies the corresponding objectives.

We say that $v \in \mathcal{C}_b(\mathcal{S})$ has a \emph{concave lift} if the map $\mathcal{W}_v$ given
by~\eqref{eq:lift} is concave on $\mathbb{R}^d_+$.
The next two results
are the core of the argument: the first one shows that $\mathcal{T}$ preserves the concavity of the lift, while the second
one shows that $\mathcal{T}$ regularises uniformly in $v$.

\begin{lemma}\label{lm:invariance}
Let $\gamma\leq 0$ and assume~\eqref{A:integrability_log}. If $v\in \mathcal{C}_b(\mathcal{S})$ has a concave lift, then so
does $\mathcal{T}v$.
\end{lemma}
\begin{proof}
First, we note that, by~\eqref{eq:lift_bounds} and Lemma~\ref{lm:reward}(iv), the map $\Phi_{\mathcal{W}_v}$ is concave. Next, recall that by Lemma~\ref{lm:cone}(i) and Lemma~\ref{lm:reward}(i), we have $\mathcal{K}(\alpha x)=\alpha \mathcal{K}(x)$ and the map $\Phi_{\mathcal{W}_v}$ is positively homogeneous. Thus, the map $\mathbb{K}\ni x\to \sup_{y\in \mathcal{K}(x)}\Phi_{\mathcal{W}_v}(y) $ is also positively homogeneous. Consequently, using~\eqref{eq:lift} and~\eqref{eq:lift_of_T}, for any $x\in \mathbb{K}$, we obtain
\[
\mathcal{W}_{\mathcal{T}v}(x)=\langle x,\mathbf{1}\rangle\, e^{\mathcal{T}v(x/\langle x, \mathbf{1}\rangle)} = \langle x,\mathbf{1}\rangle\sup_{y \in \mathcal{K}(x/\langle x,\mathbf{1}\rangle)}\Phi_{\mathcal{W}_v}(y)=\sup_{y \in \mathcal{K}(x)}\Phi_{\mathcal{W}_v}(y).
\]
Next, let $x^1,x^2\in \mathbb{K}$, $t\in [0,1]$, and $\varepsilon>0$, choose
$\varepsilon$-optimal $\hat{y}^k\in \mathcal{K}(x^k)$, $k=1,2$, and note that, by Lemma~\ref{lm:cone}(i), we have $t\hat{y}^1+(1-t)\hat{y}^2\in
\mathcal{K}(tx^1+(1-t)x^2)$. Thus, using the concavity of $\Phi_{\mathcal{W}_v}$, we get
\begin{align*}
    \mathcal{W}_{\mathcal{T}v}(tx^1+(1-t)x^2)&\geq \Phi_{\mathcal{W}_v}(t\hat{y}^1+(1-t)\hat{y}^2)\\
    & \geq t\,\Phi_{\mathcal{W}_v}(\hat{y}^1)+(1-t)\,\Phi_{\mathcal{W}_v}(\hat{y}^2) \\
    &\geq t\,\mathcal{W}_{\mathcal{T}v}(x^1)+(1-t)\,\mathcal{W}_{\mathcal{T}v}(x^2)-2\varepsilon.
\end{align*}
Letting $\varepsilon\downarrow 0$, we conclude the proof.
\end{proof}

\begin{lemma}\label{lm:apriori}
Let $\gamma\leq 0$ and assume~\eqref{A:integrability_log}. Then, for any $v\in \mathcal{C}_b(\mathcal{S})$, we have
$\mathcal{T}v\in \mathcal{C}_b(\mathcal{S})$ and, for any $\pi,\pi''\in \mathcal{S}$,
\begin{equation}\label{eq:th:Bellman_existence:unicont}
    |\mathcal{T}v(\pi)-\mathcal{T}v(\pi'')|\leq \sup_{\pi'\in \mathcal{S}}\left|\ln s(\pi,\pi')-\ln s(\pi'',\pi')\right|.
\end{equation}
Also, with $L:=\frac{\bar{c}}{(1-\bar{c})\bar{s}}$ and any $\pi,\pi''\in \mathcal{S}$, we get the bounds
\begin{equation}\label{eq:apriori}
    \spann_{\mathcal{S}}\mathcal{T}v \leq \ln \tfrac{1}{\bar{s}},
    \qquad |\mathcal{T}v(\pi)-\mathcal{T}v(\pi'')|\leq L\Vert \pi-\pi''\Vert_1,
\end{equation}
Moreover, the operator $\mathcal{T}$ is monotone, cash-additive, and non-expansive, i.e.\
$\Vert \mathcal{T}v_1-\mathcal{T}v_2\Vert \leq \Vert v_1-v_2\Vert$ for $v_1,v_2\in \mathcal{C}_b(\mathcal{S})$.
\end{lemma}
\begin{proof}
Note that~\eqref{eq:th:Bellman_existence:unicont} follows directly from~\eqref{eq:T_operator} and the fact that an absolute difference of suprema is bounded from above by the supremum of absolute differences. Hence, by Lemma~\ref{lm:s_properties}, the map $\mathcal{T}v$ is
finite and continuous, i.e.\ $\mathcal{T}v\in \mathcal{C}_b(\mathcal{S})$.

To get~\eqref{eq:apriori}, note that by Lemma~\ref{lm:s_properties}(ii) we have $\ln s\in [\ln \bar{s},0]$, which
combined with~\eqref{eq:th:Bellman_existence:unicont} gives the first bound. For the second one, using
Lemma~\ref{lm:s_properties}(ii),(iii) and the mean value theorem applied to the logarithm, we get
\[
|\ln s(\pi,\pi')-\ln s(\pi'',\pi')|\leq \tfrac{1}{\bar{s}}|s(\pi,\pi')-s(\pi'',\pi')|\leq L\Vert \pi-\pi''\Vert_1 .
\]
Finally, the monotonicity, cash-additivity, and non-expansiveness of $\mathcal{T}$ follow from the corresponding properties
of $\mu^\gamma$ listed in~\eqref{eq:mu_basic}.
\end{proof}

Now, let us fix some $\pi_0\in \mathcal{S}$ and define
\begin{multline}\label{eq:K_set}
    \mathcal{V}:=\left\{v \in \mathcal{C}_b(\mathcal{S})\colon \mathcal{W}_v \text{ is concave},\ v(\pi_0)=0,\right.\\
    \left.\spann_{\mathcal{S}} v\leq \ln \tfrac{1}{\bar{s}},\ |v(\pi)-v(\pi'')|\leq L\Vert \pi-\pi''\Vert_1\right\}.
\end{multline}

\begin{lemma}\label{lm:V_compact}
The set $\mathcal{V}$ is a non-empty, convex, and compact subset of $\mathcal{C}_b(\mathcal{S})$ equipped with the supremum
norm.
\end{lemma}
\begin{proof}
Non-emptiness follows by noting that $v\equiv 0$ belongs to $\mathcal{V}$, as $\mathcal{W}_0(\cdot)=\langle
\cdot,\mathbf{1}\rangle$ is concave.

For the convexity, let $v_1,v_2\in \mathcal{V}$ and $t\in [0,1]$. The normalisation $v\to v(\pi_0)$ is linear, while the span
seminorm and the Lipschitz property constitute convex constraints. For the concavity
of the lift, we note that, directly from~\eqref{eq:lift}, we have
\[
\mathcal{W}_{tv_1+(1-t)v_2}=(\mathcal{W}_{v_1})^{t}(\mathcal{W}_{v_2})^{1-t},
\]
since $\langle x,\mathbf{1}\rangle=\langle x,\mathbf{1}\rangle^{t}\langle x,\mathbf{1}\rangle^{1-t}$. As the map
$(u_1,u_2)\to u_1^tu_2^{1-t}$ is concave and non-decreasing on $[0,\infty)^2$, its composition with the concave maps
$\mathcal{W}_{v_1},\mathcal{W}_{v_2}$ is concave, i.e.\ $tv_1+(1-t)v_2\in \mathcal{V}$.

For the compactness, note that any $v\in \mathcal{V}$ satisfies $\Vert v\Vert\leq \ln \frac{1}{\bar{s}}$, due to
$v(\pi_0)=0$ and the span bound, and is Lipschitz with the constant $L$. Thus, the family $\mathcal{V}$ is uniformly
bounded and equicontinuous, so it is relatively compact by the Arzel\`{a}-Ascoli theorem. Also, $\mathcal{V}$ is closed
under the uniform convergence: this is clear for the three explicit constraints, while the concavity of the lift is
preserved, since the uniform convergence $v_n\to v$ implies the pointwise convergence $\mathcal{W}_{v_n}\to \mathcal{W}_v$
and the pointwise limits of concave maps are concave.
\end{proof}

We are now ready to prove the main existence result of this section.

\begin{theorem}\label{th:existence_concave}
Let $\gamma\leq 0$. Also, assume~\eqref{B:increments},~\eqref{B:waiting_times}, and~\eqref{A:integrability_log}. Then, there
exists a solution $(\bar{v}^\gamma,\bar\lambda^\gamma)\in \mathcal{C}_b(\mathcal{S})\times \mathbb{R}$
to~\eqref{eq:Bellman_log} such that
\begin{enumerate}[label=(\roman*)]
    \item the lift $\mathcal{W}_{\bar{v}^\gamma}$ is concave and positively homogeneous on $\mathbb{R}^d_+$;
    \item after the normalisation $\max_{\mathcal{S}}\bar{v}^\gamma=0$, we have $\bar{s}\langle x,\mathbf{1}\rangle\leq
    \mathcal{W}_{\bar{v}^\gamma}(x)\leq \langle x,\mathbf{1}\rangle$, $x\in \mathbb{R}^d_+$;
    \item $\bar{v}^\gamma$ is Lipschitz with the constant $L$ and $\spann_{\mathcal{S}}\bar{v}^\gamma\leq \ln \frac{1}{\bar{s}}$;
    \item $\bar\lambda^\gamma=\frac{1}{\mathbb{E}[\eta_1]}(\mathcal{T}\bar{v}^\gamma)(\pi_0)$.
\end{enumerate}
\end{theorem}
\begin{proof}
Let us define the normalised operator $\widehat{\mathcal{T}}v:=\mathcal{T}v-(\mathcal{T}v)(\pi_0)$, $v \in
\mathcal{C}_b(\mathcal{S})$, and recall the set $\mathcal{V}$ given by~\eqref{eq:K_set}.

First, we show that $\widehat{\mathcal{T}}$ maps $\mathcal{V}$ into itself. Indeed, let $v\in \mathcal{V}$. Then, by
Lemma~\ref{lm:invariance}, the lift $\mathcal{W}_{\mathcal{T}v}$ is concave; also, subtracting a constant from
$\mathcal{T}v$ multiplies its lift by a positive constant, so the lift of $\widehat{\mathcal{T}}v$ is concave as well.
Next, the normalisation $\widehat{\mathcal{T}}v(\pi_0)=0$ holds by the construction, while the two remaining constraints
follow from~\eqref{eq:apriori}, as the span seminorm and the Lipschitz constant are not changed by subtracting a constant. 

Second, the operator $\widehat{\mathcal{T}}$ is continuous with respect to the supremum norm. Indeed, using the
non-expansiveness of $\mathcal{T}$ stated in Lemma~\ref{lm:apriori}, for any $v_1,v_2 \in \mathcal{C}_b(\mathcal{S})$ we get
\[
\Vert \widehat{\mathcal{T}}v_1-\widehat{\mathcal{T}}v_2\Vert\leq \Vert \mathcal{T}v_1-\mathcal{T}v_2\Vert+
|(\mathcal{T}v_1-\mathcal{T}v_2)(\pi_0)|\leq 2\Vert v_1-v_2\Vert.
\]

Combining this with Lemma~\ref{lm:V_compact} and using the Schauder fixed point theorem (see, e.g., Theorem 3.2 in~\cite{Bon1962} or Theorem 2.A in~\cite{Zei1986}), we may find $\bar{v}^\gamma\in \mathcal{V}$ such that
$\widehat{\mathcal{T}}\bar{v}^\gamma=\bar{v}^\gamma$, i.e.
\[
\mathcal{T}\bar{v}^\gamma=\bar{v}^\gamma+(\mathcal{T}\bar{v}^\gamma)(\pi_0).
\]
Thus, setting $\bar\lambda^\gamma:=\frac{1}{\mathbb{E}[\eta_1]}(\mathcal{T}\bar{v}^\gamma)(\pi_0)$, which is well defined
since $0<\mathbb{E}[\eta_1]<\infty$ by Assumption~\eqref{B:waiting_times}, we recover~\eqref{eq:fixed_point}, i.e.\ the pair
$(\bar{v}^\gamma,\bar\lambda^\gamma)$ is a solution to~\eqref{eq:Bellman_log}. Properties (i), (iii), and (iv) follow from
$\bar{v}^\gamma\in \mathcal{V}$ and the definition of $\bar\lambda^\gamma$. 
Finally, (ii) follows from~\eqref{eq:lift_bounds} combined with (iii) and the normalisation $\max_{\mathcal{S}}\bar{v}^\gamma=0$.
\end{proof}


\subsection{Uniqueness of the maximiser}\label{S:unique_maximiser_subsec}

In this section, we show that, for a fixed solution to the Bellman equation~\eqref{eq:Bellman_log}, the supremum on the right-hand side is attained
at exactly one point and that the corresponding maximiser depends continuously on the current portfolio position. This
justifies referring to \emph{the} optimal rebalancing in the strategy~\eqref{eq:th:verification:strategy} and shows that no
measurable selection argument is needed in the verification theorem.

Throughout this section we fix a solution $(\bar{v}^\gamma,\bar\lambda^\gamma)\in \mathcal{C}_b(\mathcal{S})\times
\mathbb{R}$ to~\eqref{eq:Bellman_log} such that the lift $\mathcal{W}:=\mathcal{W}_{\bar{v}^\gamma}$ is concave; by
Theorem~\ref{th:existence_concave} at least one such solution exists. In particular, by~\eqref{eq:lift_bounds} there are
constants $0<c_1\leq c_2<\infty$ such that
\begin{equation}\label{eq:W_bounds}
    c_1\langle x,\mathbf{1}\rangle\leq \mathcal{W}(x)\leq c_2\langle x,\mathbf{1}\rangle, \quad x \in \mathbb{R}^d_+.
\end{equation}
We also write $\Phi:=\Phi_{\mathcal{W}}$ and $\Psi:=\ln \Phi$, and, for a fixed $\pi\in \mathcal{S}$, we define the sets of
maximisers in the weight and the capital coordinates by
\begin{align}\label{eq:argmax_sets}
    \widehat\Pi(\pi)&:=\argmax_{\pi'\in \mathcal{S}}\left(\ln s(\pi,\pi')+\Psi_{\bar{v}^\gamma}(\pi')\right), &
    \widehat{X}(\pi)&:=\argmax_{y\in \mathcal{K}(\pi)}\Psi(y),
\end{align}
respectively. Note that, by~\eqref{eq:mu_basic}, the set $\widehat\Pi(\pi)$ coincides with the set of maximisers on the
right-hand side of~\eqref{eq:Bellman_log}, since the two expressions differ by the constant
$-\bar\lambda^\gamma\mathbb{E}[\eta_1]$.

We start by showing that the two sets in~\eqref{eq:argmax_sets} are in a one-to-one correspondence. 

\begin{lemma}\label{lm:argmax_correspondence}
Let $\gamma\leq 0$ and assume~\eqref{A:integrability_log}. Then, for any $\pi\in \mathcal{S}$, the map $\iota_\pi$ from
Lemma~\ref{lm:cone}(iv) restricted to $\widehat\Pi(\pi)$ is a bijection onto $\widehat{X}(\pi)$, with the inverse
given by $y\to y/\langle y,\mathbf{1}\rangle$. In particular, the set $\widehat\Pi(\pi)$ is a singleton if and only if
$\widehat{X}(\pi)$ is a singleton.
\end{lemma}
\begin{proof}
Let $J_\pi(\pi'):=\ln s(\pi,\pi')+\Psi_{\bar{v}^\gamma}(\pi')$, $\pi'\in \mathcal{S}$. We combine three facts. First, by
Lemma~\ref{lm:cone}(iv), the map $\iota_\pi$ is a bijection from $\mathcal{S}$ onto $\Sigma(\pi)$. Second, by
Lemma~\ref{lm:reward}(iii), for any $\pi'\in \mathcal{S}$ we have
\[
\Psi(\iota_\pi(\pi'))=\ln \Phi_{\mathcal{W}}(\iota_\pi(\pi'))=\ln s(\pi,\pi')+\Psi_{\bar{v}^\gamma}(\pi')=J_\pi(\pi'),
\]
i.e.\ the map $\iota_\pi$ transports the objective $J_\pi$ on $\mathcal{S}$ onto the objective $\Psi$ restricted to
$\Sigma(\pi)$, without any distortion. Third, by Lemma~\ref{lm:reward}(ii), we have
\[
\widehat{X}(\pi)=\argmax_{y\in \mathcal{K}(\pi)}\Psi(y)=\argmax_{y\in \Sigma(\pi)}\Psi(y).
\]
Combining these facts, we get that $\pi'$ maximises $J_\pi$ over $\mathcal{S}$ if and only if $\iota_\pi(\pi')$ maximises
$\Psi$ over $\Sigma(\pi)$, i.e.\ if and only if $\iota_\pi(\pi')\in \widehat{X}(\pi)$. Since $\iota_\pi$ is a bijection from
$\mathcal{S}$ onto $\Sigma(\pi)$, its restriction to $\widehat\Pi(\pi)$ is a bijection onto $\widehat{X}(\pi)$.
\end{proof}

\begin{lemma}\label{lm:argmax_convex}
Let $\gamma\leq 0$ and assume~\eqref{A:integrability_log}. Then, for any $\pi \in \mathcal{S}$, the set
$\widehat{X}(\pi)$ is non-empty, compact, convex, and contained in $\Sigma(\pi)$. Moreover, $\langle
y,\mathbf{1}\rangle\geq \bar{s}$ for any $y \in \widehat{X}(\pi)$.
\end{lemma}
\begin{proof}
Non-emptiness, compactness, and the inclusion $\widehat{X}(\pi)\subset \Sigma(\pi)$ follow from
Lemma~\ref{lm:reward}(ii), while the bound $\langle y,\mathbf{1}\rangle\geq \bar{s}$ is a consequence of
Lemma~\ref{lm:cone}(iii). For the convexity, we note that $\Phi=\Phi_{\mathcal{W}}$ is concave, by
Lemma~\ref{lm:reward}(iv), and $\mathcal{K}(\pi)$ is convex, by Lemma~\ref{lm:cone}(i). Thus, the set of maximisers of the
concave map $\Phi$ over the convex set $\mathcal{K}(\pi)$ is convex, and so is
$\widehat{X}(\pi)=\argmax_{\mathcal{K}(\pi)}\ln \Phi$, since $\ln$ is strictly increasing.
\end{proof}

The following result is the key step of the argument: any two optimal rebalancings lead to the same wealth in (almost) every
state of the world. Let us stress that no assumption on the support of the growth rates is required here.

\begin{proposition}\label{pr:same_wealth}
Let $\gamma\leq 0$ and assume~\eqref{A:integrability_log}. Then, for any $\pi \in \mathcal{S}$ and $y^1,y^2\in
\widehat{X}(\pi)$, we have
\[
\mathcal{W}(y^1\cdot R(\eta_1))=\mathcal{W}(y^2\cdot R(\eta_1)), \quad \mathbb{P}\text{-a.s.}
\]
\end{proposition}
\begin{proof}
Let $y^m:=\frac{1}{2}(y^1+y^2)$ and note that, by Lemma~\ref{lm:argmax_convex}, we have $y^m\in \widehat{X}(\pi)$. Also, set
\[
\mathcal{W}_k:=\mathcal{W}(y^k\cdot R(\eta_1)) \quad \text{and}\quad Z_k:=\ln \mathcal{W}_k, \quad k\in \{1,2,m\},
\]
and denote by $M$ the common value $\Psi(y^1)=\Psi(y^2)=\Psi(y^m)=\max_{\mathcal{K}(\pi)}\Psi$. Also, recalling~\eqref{eq:lift} and~\eqref{eq:reward}, we get $\mu^\gamma[Z_k]=\Psi(y^k)=M$, for $k \in \{1,2,m\}$. Straightforward verification shows that, under~\eqref{A:integrability_log}, we have that $M$ is finite.

Next, since the map $y\to y\cdot R(\eta_1)$ is linear and $\mathcal{W}$ is
concave, we get, pointwise,
\begin{equation}\label{eq:pr:same_wealth:0}
    \mathcal{W}_m\geq \tfrac{1}{2}(\mathcal{W}_1+\mathcal{W}_2)\geq \sqrt{\mathcal{W}_1\mathcal{W}_2},
\end{equation}
where the second inequality is the inequality between the arithmetic and the geometric means. Taking the logarithms, we
obtain
\begin{equation}\label{eq:pr:same_wealth:1}
    Z_m\geq \tfrac{1}{2}(Z_1+Z_2), \quad \mathbb{P}\text{-a.s.}
\end{equation}
Also, using the concavity of $\mu^\gamma$, see~\eqref{eq:mu_basic}, and then the
monotonicity of $\mu^\gamma$ combined with~\eqref{eq:pr:same_wealth:1}, we get
\[
M=\tfrac{1}{2}\mu^\gamma[Z_1]+\tfrac{1}{2}\mu^\gamma[Z_2]\leq
\mu^\gamma\left[\tfrac{1}{2}(Z_1+Z_2)\right]\leq\mu^\gamma[Z_m]=M.
\]
Consequently, both inequalities above are in fact equalities. In particular, using~\eqref{eq:pr:same_wealth:1} and $\mu^\gamma\left[\tfrac{1}{2}(Z_1+Z_2)\right]=\mu^\gamma[Z_m]$, we get
\[
Z_m=\tfrac{1}{2}(Z_1+Z_2), \quad \mathbb{P}\text{-a.s.},
\]
i.e.\ $\mathcal{W}_m=\sqrt{\mathcal{W}_1\mathcal{W}_2}$ a.s. In fact, recalling~\eqref{eq:pr:same_wealth:0}, we have
\[
\sqrt{\mathcal{W}_1\mathcal{W}_2}= \tfrac{1}{2}(\mathcal{W}_1+\mathcal{W}_2), \quad \mathbb{P}\text{-a.s.},
\]
and the equality in the inequality between the arithmetic and the geometric means forces $\mathcal{W}_1=\mathcal{W}_2$ a.s.,
which concludes the proof.
\end{proof}

We are now ready to state the main result of this section.

\begin{theorem}\label{th:unique_maximiser}
Let $\gamma\leq 0$ and assume~\eqref{A:integrability_log} and~\eqref{A:support}. Then, for a fixed solution $(\bar{v}^\gamma, \bar\lambda^\gamma)\in\mathcal{C}_b(\mathcal{S})\times \mathbb{R}$ to~\eqref{eq:Bellman_log} with concave lift and any $\pi \in \mathcal{S}$, the
supremum on the right-hand side of~\eqref{eq:Bellman_log} is attained at exactly one point $\hat\pi(\pi)\in \mathcal{S}$.
Moreover, the map $\mathcal{S}\ni \pi\to \hat\pi(\pi)\in \mathcal{S}$ is continuous.
\end{theorem}
\begin{proof}
Let us fix $\pi\in \mathcal{S}$, take $y^1,y^2\in \widehat{X}(\pi)$, and define
\[
\varphi(x):=\mathcal{W}(y^1\cdot x)-\mathcal{W}(y^2\cdot x), \quad x\in \mathbb{R}^d_+.
\]
Note that $\varphi$ is continuous and, by the positive homogeneity of $\mathcal{W}$, positively homogeneous. Also, by
Proposition~\ref{pr:same_wealth}, we have $\varphi(R(\eta_1))=0$ a.s.

Now, we transfer $\varphi$ to the simplex $\mathcal{S}$. More specifically, since $\varphi$ is positively homogeneous and $\langle x,\mathbf{1}\rangle>0$ for
$x\in \mathbb{K}$, we may write $\varphi(x)=\langle x,\mathbf{1}\rangle\, g(x/\langle x,\mathbf{1}\rangle)$, where
$g:=\varphi|_{\mathcal{S}}\in \mathcal{C}_b(\mathcal{S})$. Thus, for $x\in \mathbb{K}$ we have $\varphi(x)=0$ if and only if
$g(x/\langle x,\mathbf{1}\rangle)=0$. In particular, denoting the (random) direction of the growth rates by
\[
\mathcal{R}:=R(\eta_1)/\langle R(\eta_1),\mathbf{1}\rangle \in \mathcal{S},
\]
the property $\varphi(R(\eta_1))=0$ a.s.\ is equivalent to $g(\mathcal{R})=0$ a.s.

Let us fix some $\pi^\circ\in \mathcal{S}_0$ and consider the map
$m\colon \mathcal{S}\to \mathcal{S}$ given by $m(u):=\frac{\pi^\circ\cdot u}{\langle \pi^\circ,u\rangle}$. Since all the
coordinates of $\pi^\circ$ are strictly positive, the map $m$ is a homeomorphism of $\mathcal{S}$, with the inverse given by
$m^{-1}(w)=\frac{w/\pi^\circ}{\langle w/\pi^\circ,\mathbf{1}\rangle}$, where the division is understood componentwise. Also,
by a direct computation,
\[
G(\pi^\circ, R(\eta_1))=\frac{\pi^\circ\cdot R(\eta_1)}{\langle \pi^\circ,R(\eta_1)\rangle}=\frac{\pi^\circ\cdot
\mathcal{R}}{\langle \pi^\circ,\mathcal{R}\rangle}=m(\mathcal{R}), \quad \text{i.e.}\quad
\mathcal{R}=m^{-1}(G(\pi^\circ,R(\eta_1))).
\]
By Assumption~\eqref{A:support}, the random variable $G(\pi^\circ,R(\eta_1))$ with positive probability takes values in any non-empty open subset of
$\mathcal{S}$ ; since $m^{-1}$ is a homeomorphism, the same is true for $\mathcal{R}$.

Now, we show that $g\equiv 0$. For the contradiction, suppose that $g(u_0)\neq 0$ for some $u_0\in \mathcal{S}$. Then, by the continuity of $g$, there
is a relatively open set $V\subset \mathcal{S}$ containing $u_0$ such that $g\neq 0$ on $V$. On the other hand, we have
$\mathbb{P}[\mathcal{R}\in V]>0$, which contradicts $g(\mathcal{R})=0$ a.s. Thus, $g\equiv 0$ on $\mathcal{S}$, which combined with the continuity of $\varphi$, shows $\varphi\equiv 0$ on $\mathbb{R}^d_+$, i.e.
\begin{equation}\label{eq:th:unique_maximiser:1}
    \mathcal{W}(y^1\cdot x)=\mathcal{W}(y^2\cdot x), \quad x \in \mathbb{R}^d_+.
\end{equation}
Let us fix $i\in \{1,\ldots,d\}$ and apply~\eqref{eq:th:unique_maximiser:1} with $x:=e_i$, the $i$th vector of the canonical
basis; this is admissible, since~\eqref{eq:th:unique_maximiser:1} holds on the whole orthant, including its vertices. Using
the positive homogeneity of $\mathcal{W}$, we get $\mathcal{W}(y^k\cdot e_i)=\mathcal{W}(y^k_ie_i)=y^k_i\mathcal{W}(e_i)$,
$k=1,2$, so that
\[
y^1_i\mathcal{W}(e_i)=y^2_i\mathcal{W}(e_i), \quad \text{where } \mathcal{W}(e_i)\geq c_1>0,
\]
by~\eqref{eq:W_bounds}. Hence, $y^1_i=y^2_i$ and, as $i$ was arbitrary, $y^1=y^2$. This shows that $\widehat{X}(\pi)$ is a
singleton and, by Lemma~\ref{lm:argmax_correspondence}, so is $\widehat\Pi(\pi)$; we denote its unique element by
$\hat\pi(\pi)$.

It remains to show the continuity of $\hat\pi$. Let $\pi_n\to \pi$ in $\mathcal{S}$ and suppose that $\hat\pi(\pi_n)\not\to
\hat\pi(\pi)$. Then, by the compactness of $\mathcal{S}$, we may find a subsequence $(\pi_{n_k})$ and $\tilde\pi\neq
\hat\pi(\pi)$ such that $\hat\pi(\pi_{n_k})\to \tilde\pi$ as $k\to\infty$. Recalling that
$J_\pi(\pi')=\ln s(\pi,\pi')+\Psi_{\bar{v}^\gamma}(\pi')$ is jointly continuous in $(\pi,\pi')$, by Lemma~\ref{lm:s_properties} and
Lemma~\ref{lm:Feller}, and passing to the limit in the inequality $J_{\pi_{n_k}}(\hat\pi(\pi_{n_k}))\geq
J_{\pi_{n_k}}(\pi')$, $\pi'\in \mathcal{S}$, we get $J_{\pi}(\tilde\pi)\geq J_{\pi}(\pi')$ for any $\pi'\in \mathcal{S}$.
Thus, $\tilde\pi$ is a maximiser at $\pi$, which contradicts the uniqueness shown above.
\end{proof}

\begin{remark}
It should be noted that some support condition is necessary for the uniqueness of the maximiser. Indeed, for $d=3$, if $R_1(\eta_1)=R_2(\eta_1)$ a.s.\ and $c_1=h_1=c_2=h_2=0$, assets $1$ and $2$ are indistinguishable and
freely exchangeable, and $\widehat X(\pi)$ contains a segment. Consequently, by Lemma~\ref{lm:argmax_correspondence}, maximisers are not unique.
\end{remark}


\subsection{Uniqueness of a solution to the Bellman equation}\label{SS:unique_solution}

Based on Theorem~\ref{th:verification} we know that the constant $\bar\lambda^\gamma$ from Equation~\eqref{eq:Bellman_log} is uniquely determined as the optimal value of the underlying control problem. Also, in Theorem~\ref{th:unique_maximiser}, we showed that, under the full support condition, the maximiser for a fixed solution to the Bellman equation is unique. In this section, we finally show that a function solving~\eqref{eq:Bellman_log} is unique up to an additive constant.

Before we proceed, let us state an additional condition used in this section. In principle, the assumption will ensure that all the assets are present in the optimal portfolio rebalancing.

\begin{enumerate}
\item[(\namedlabel{A:growth_gamma}{$\mathcal{A}5$})] (Growth)
For any $\gamma\le 0$, $i\in\{1,\dots,d\}$, and every $\pi'\in\mathcal{S}$ with $\pi'_i=0$, we have
\begin{equation}\label{eq:growth_gamma}
    \frac{\mathbb{E}\!\left[R_i(\eta_1)\,\langle\pi',R(\eta_1)\rangle^{\gamma-1}\right]}{\mathbb{E}\!\left[\langle\pi',R(\eta_1)\rangle^{\gamma}\right]}\ >\ \bar s^{\,(\gamma-1)}\,(1+C),
\end{equation}
where $C:=2\Lambda/\bar s$ and $\Lambda:=\frac{\bar{c}}{1-\bar{c}}$ is the Lipschitz constant of $s(\pi,\cdot)$ from Lemma~\ref{lm:s_properties}(iii)).
\end{enumerate}
It should be noted that, for $\gamma=0$, Assumption~\eqref{A:growth_gamma} states that the ratio of the one-period growth of the $i$th asset $R_i(\eta_1)$ to the $i$-free portfolio growth $\langle \pi', R(\eta_1)\rangle$ is on average bounded from below by a sufficiently big constant. For $\gamma<0$, the assumption accounts for the non-linearity of the entropic utility.

Now, we state a technical result showing that, under~\eqref{A:growth_gamma}, all the assets are used in the Bellman equation-induced optimal rebalancing. Note that, in the proposition, we use the existence of a solution to the Bellman equation with the concave lift, which is guaranteed by Theorem~\ref{th:existence_concave}.

\begin{proposition}\label{pr:interior}
Let $\gamma\le 0$, assume \eqref{A:integrability_log} and \eqref{A:growth_gamma}, and let $(\bar{v}^\gamma, \bar\lambda^\gamma)\in\mathcal{C}_b(\mathcal{S})\times \mathbb{R}$ solve the Bellman equation~\eqref{eq:Bellman_log} with concave lift $\mathcal{W}_{\bar v^\gamma}$. Then, for every $\pi\in\mathcal{S}$, every maximiser $\pi^\gamma(\pi)$ of the right-hand side of~\eqref{eq:Bellman_log} lies in the interior $\mathcal{S}_0$.
\end{proposition}

\begin{proof}
Fix $\pi\in\mathcal{S}$ and set
\[
    \Phi(\pi'):=\ln s(\pi,\pi')+\mu^\gamma\!\big[\ln\mathcal{W}_{\bar v^\gamma}(\pi'\cdot R(\eta_1))\big],\qquad \pi'\in\mathcal{S}.
\]
By~\eqref{eq:lift} and Lemma~\ref{lm:Feller}, the map $\Phi$ is continuous, attains its supremum on the compact set $\mathcal{S}$, and the maximisers of $\Phi$ correspond to the maximisers at the right-hand side of~\eqref{eq:Bellman_log}. Thus, it suffices to show that no boundary point maximises $\Phi$. To do this, let $\pi^\circ\in\partial\mathcal{S}$ satisfy $\pi^\circ_i=0$ for some $i$, and for $\varepsilon\in(0,1)$ put $\pi'_\varepsilon:=(1-\varepsilon)\pi^\circ+\varepsilon e_i$, where $e_i$ is the $i$th vertex of $\mathcal{S}$. Write $x:=\pi^\circ\cdot R(\eta_1)$, so $x_i=0$ and $\langle x,\mathbf{1}\rangle=\langle\pi^\circ,R(\eta_1)\rangle$. Also, set $Y:=R_i(\eta_1)/\langle\pi^\circ,R(\eta_1)\rangle\ge0$. In the following, we show that for a sufficiently small $\varepsilon$, we have $\Phi(\pi'_\varepsilon)>\Phi(\pi^\circ)$, which shows that no boundary point is optimal.

First, we establish a lower bound for $\ln\mathcal{W}_{\bar v^\gamma}\big(\pi'_\varepsilon\cdot R(\eta_1)\big)-\ln\mathcal{W}_{\bar v^\gamma}(x) $. Since $\pi'_\varepsilon\cdot R(\eta_1)=(1-\varepsilon)\,x+\varepsilon\,\big(e_i\cdot R(\eta_1)\big)=(1-\varepsilon)\,x+\varepsilon\,R_i(\eta_1)\,e_i$ is a convex combination of $x$ and $R_i(\eta_1)e_i$, the concavity of $\mathcal{W}_{\bar v^\gamma}$ (see Theorem~\ref{th:existence_concave}(i))  gives
\begin{equation}\label{eq:pr:interior:1}
    \mathcal{W}_{\bar v^\gamma}\big(\pi'_\varepsilon\cdot R(\eta_1)\big)
    \ \ge\ (1-\varepsilon)\,\mathcal{W}_{\bar v^\gamma}(x)+\varepsilon\,\mathcal{W}_{\bar v^\gamma}\big(R_i(\eta_1)e_i\big), 
\end{equation}
and the $1$-homogeneity of $\mathcal{W}_{\bar v^\gamma}$ turns the last term into $\varepsilon R_i(\eta_1)\mathcal{W}_{\bar v^\gamma}(e_i)$. Using the facts that $\mathcal{W}_{\bar v^\gamma}(x)=\langle\pi^\circ,R(\eta_1)\rangle\,e^{\bar v^\gamma(G(\pi^\circ,R(\eta_1)))}>0$ and $\mathcal{W}_{\bar v^\gamma}(e_i)=e^{\bar v(e_i)}$, we get
\[
    \frac{\mathcal{W}_{\bar v^\gamma}(e_i)}{\mathcal{W}_{\bar v^\gamma}(x)}
    =\frac{e^{\,\bar v^\gamma(e_i)-\bar v^\gamma(G(\pi^\circ,R(\eta_1)))}}{\langle\pi^\circ,R(\eta_1)\rangle}
    \ \ge\ \frac{e^{-\spann(\bar v^\gamma)}}{\langle\pi^\circ,R(\eta_1)\rangle}
    =\frac{\kappa}{\langle\pi^\circ,R(\eta_1)\rangle},
\]
where $\kappa:=e^{-\spann(\bar v^\gamma)}>0$. Thus, from~\eqref{eq:pr:interior:1}, we obtain $\mathcal{W}_{\bar v^\gamma}(\pi'_\varepsilon\cdot R(\eta_1))/\mathcal{W}_{\bar v^\gamma}(x)\ge(1-\varepsilon)+\varepsilon\kappa Y=:L_\varepsilon>0$. Taking logarithms, we get
\begin{equation}\label{eq:ratio_unified}
    \ln\mathcal{W}_{\bar v^\gamma}\big(\pi'_\varepsilon\cdot R(\eta_1)\big)-\ln\mathcal{W}_{\bar v^\gamma}(x)\ \ge\ \ln L_\varepsilon .
\end{equation}

Second, we provide a lower bound for $\mu^\gamma[\ln\mathcal{W}_{\bar v^\gamma}\big(\pi'_\varepsilon\cdot R(\eta_1)\big)]-\mu^\gamma[\ln\mathcal{W}_{\bar v^\gamma}(x) ]$.  The argument is based on the change-of-measure technique. More specifically, let $U:=\ln\mathcal{W}_{\bar v^\gamma}(x)$, so $e^{\gamma U}=\left(\mathcal{W}_{\bar v^\gamma}(x)\right)^\gamma$. Also, note that, for any $\gamma\leq 0$, we have
\[
0<e^{\gamma U} = \left(\mathcal{W}_{\bar v^\gamma}(x)\right)^\gamma\le e^{-\gamma\Vert \bar v^\gamma \Vert}\langle\pi^\circ,R(\eta_1)\rangle^\gamma\le e^{-\gamma\Vert \bar v^\gamma \Vert}e^{\gamma\min_j r_j(\eta_1)}
\]
and, by~\eqref{A:integrability_log}, the upper bound is integrable. Thus, we may set $D:=\mathbb{E}[\left(\mathcal{W}_{\bar v^\gamma}(x)\right)^\gamma]$ and define a probability measure $\widehat{\mathbb{P}}$ by $\tfrac{d\widehat{\mathbb{P}}}{d\mathbb{P}}:=\mathcal{W}_{\bar v^\gamma}(x)^\gamma/D$; note that $\widehat{\mathbb{P}}=\mathbb{P}$ for $\gamma=0$. Next, for $\gamma<0$ and any sufficiently integrable $Z$, we set $\mu^\gamma_{\widehat{\mathbb{P}}}[Z]:=\tfrac1\gamma\ln\mathbb{E}_{\widehat{\mathbb{P}}}\big[e^{\gamma Z}\big]$ and note that
\begin{equation}\label{eq:tilt_identity}
\mu^\gamma[U+Z]=\tfrac1\gamma\ln\mathbb{E}[e^{\gamma U}e^{\gamma Z}]=\tfrac1\gamma\ln(\mathbb{E}[e^{\gamma U}]\,\mathbb{E}_{\widehat{\mathbb{P}}}[e^{\gamma Z}])=\mu^\gamma[U]+\mu^\gamma_{\widehat{\mathbb{P}}}[Z].
\end{equation}
In fact, the identity is also true for $\gamma=0$ with $\mu^0_{\widehat{\mathbb{P}}}[Z]:=\mathbb{E}[Z]$ by the additivity of the expectation. Thus, for any $\gamma\leq 0$, using the monotonicity of the entropic utility, inequality~\eqref{eq:ratio_unified}, and the identity~\eqref{eq:tilt_identity} with $Z=\ln L_\varepsilon$, we obtain
\begin{align}\label{eq:Delta_unified}
    \Delta_\varepsilon&:=\mu^\gamma\!\big[\ln\mathcal{W}_{\bar v^\gamma}(\pi'_\varepsilon\cdot R(\eta_1))\big]-\mu^\gamma[\ln\mathcal{W}_{\bar v^\gamma}(x) ] \nonumber\\
    &=\mu^\gamma\!\big[\ln\mathcal{W}_{\bar v^\gamma}(\pi'_\varepsilon\cdot R(\eta_1))-U+U\big]-\mu^\gamma[U]
    \nonumber\\
    &\ge\ \mu^\gamma\!\big[\ln L_\varepsilon + U\big]-\mu^\gamma[U]
    =\mu^\gamma_{\widehat{\mathbb{P}}}\!\big[\ln L_\varepsilon\big].
\end{align}

Third, we provide a lower bound for $\frac{\Delta_\varepsilon}{\varepsilon}$ with small $\varepsilon>0$. To do this, for any $K>0$ define a bounded random variable $W_K:=\kappa(Y\wedge K)-1$. As $Y\ge Y\wedge K$ and $\mu^\gamma_{\widehat{\mathbb{P}}}$ is monotone, for any $K>0$, we get
\[
\mu^\gamma_{\widehat{\mathbb{P}}}[\ln L_\varepsilon] = \mu^\gamma_{\widehat{\mathbb{P}}}[\ln((1-\varepsilon)+\varepsilon\kappa Y) ]\ge \mu^\gamma_{\widehat{\mathbb{P}}}[\ln(1+\varepsilon W_K)]=:g_K(\varepsilon).
\]
Also, the map $\varepsilon \to g_K(\varepsilon)$ is differentiable at $0$ and, by direct computation, we get $g_K(0)=0$ and  $g_K'(0)=\mathbb{E}_{\widehat{\mathbb{P}}}[W_K]=\kappa\mathbb{E}_{\widehat{\mathbb{P}}}[Y\wedge K]-1$. Hence $\liminf_{\varepsilon\downarrow0}\Delta_\varepsilon/\varepsilon\ge g_K'(0)$ for every $K$. Since the left-hand side is independent of $K$, taking the supremum over $K>0$ and using monotone convergence, we obtain $\mathbb{E}_{\widehat{\mathbb{P}}}[Y\wedge K]\uparrow\mathbb{E}_{\widehat{\mathbb{P}}}[Y]$, hence
\begin{equation}\label{eq:liminf_unified}
    \liminf_{\varepsilon\downarrow0}\frac{\Delta_\varepsilon}{\varepsilon}\ \ge\ \kappa\,\mathbb{E}_{\widehat{\mathbb{P}}}[Y]-1 .
\end{equation}

Fourth, we provide a lower bound for $\kappa\,\mathbb{E}_{\widehat{\mathbb{P}}}[Y]$ under~\eqref{A:growth_gamma}. Recalling that $\mathcal{W}_{\bar v^\gamma}(x)=\langle\pi^\circ,R(\eta_1)\rangle\,e^{\bar v^\gamma(G(\pi^\circ,R(\eta_1)))}$, for any $\gamma\le0$, we obtain
\[
    e^{\gamma\sup \bar v^\gamma }\langle\pi^\circ,R(\eta_1)\rangle^{\gamma}\ \le\ \mathcal{W}_{\bar v^\gamma}(x)^\gamma\ \le\ e^{\gamma\inf\bar v^\gamma }\langle\pi^\circ,R(\eta_1)\rangle^{\gamma}.
\]
Using these inequalities and the fact that $\langle\pi^\circ,R(\eta_1)\rangle^{\gamma}Y=R_i(\eta_1)\langle\pi^\circ,R(\eta_1)\rangle^{\gamma-1}$, and setting $    \rho:=\frac{\mathbb{E}[R_i(\eta_1)\langle\pi^\circ,R(\eta_1)\rangle^{\gamma-1}]}{\mathbb{E}[\langle\pi^\circ,R(\eta_1)\rangle^{\gamma}]}$, we get
\begin{align*}
    \mathbb{E}_{\widehat{\mathbb{P}}}[Y] = \frac{\mathbb{E}[Y\left(\mathcal{W}_{\bar v^\gamma}(x)\right)^\gamma]}{\mathbb{E}[\left(\mathcal{W}_{\bar v^\gamma}(x)\right)^\gamma] }&=\frac{\mathbb{E}[R_i(\eta_1)\langle\pi^\circ,R(\eta_1)\rangle ^{-1}\left(\mathcal{W}_{\bar v^\gamma}(x)\right)^\gamma]}{\mathbb{E}[\left(\mathcal{W}_{\bar v^\gamma}(x)\right)^\gamma] }\\
    &\ge\ \frac{e^{\gamma\sup\bar v}}{e^{\gamma\inf\bar v}}\,
    \frac{\mathbb{E}[R_i(\eta_1)\langle\pi^\circ,R(\eta_1)\rangle^{\gamma-1}]}{\mathbb{E}[\langle\pi^\circ,R(\eta_1)\rangle^{\gamma}]}
    \ =\ e^{\gamma\spann\bar v^\gamma}\,\rho.
\end{align*}
Multiplying by $\kappa=e^{-\spann\bar v^\gamma}$ and recalling that by Lemma~\ref{lm:apriori} we have $\spann \bar v^\gamma\le  \ln(1/\bar s)$ so that $e^{(\gamma-1)\spann\bar v^\gamma}\ge \bar s^{(1-\gamma)}$), we obtain
\[
    \kappa\,\mathbb{E}_{\widehat{\mathbb{P}}}[Y]\ \ge\ e^{(\gamma-1)\spann\bar v^\gamma}\,\rho\ \ge\ \bar s^{\,(1-\gamma)}\,\rho .
\]
By Assumption~\eqref{A:growth_gamma}, we have $\rho>\bar s^{\,(\gamma-1)}(1+C)$, and therefore
\begin{equation}\label{eq:pr:interior:2}
    \kappa\,\mathbb{E}_{\widehat{\mathbb{P}}}[Y]\ >\ \bar s^{\,(1-\gamma)}\,\bar s^{\,(\gamma-1)}\,(1+C)\ =\ 1+C .
\end{equation}

Finally, we show that $\Phi(\pi'_\varepsilon)>\Phi(\pi^\circ)$ for a sufficiently small $\varepsilon>0$. 
By Lemma~\ref{lm:s_properties}, the map $s(\pi,\cdot)$ is Lipschitz with constant $\Lambda$, and $s\ge\bar s>0$, so $\ln s(\pi,\cdot)$ is Lipschitz with constant $\Lambda/\bar s$. Also, since $\|\pi'_\varepsilon-\pi^\circ\|_1=\varepsilon\|e_i-\pi^\circ\|_1\le2\varepsilon$, we have
\[
    \frac{\ln s(\pi,\pi'_\varepsilon)-\ln s(\pi,\pi^\circ)}{\varepsilon}\ \ge\ -\frac{2\Lambda}{\bar s}=-C .
\]
Combining this with $\Phi(\pi'_\varepsilon)-\Phi(\pi^\circ)=(\ln s(\pi,\pi'_\varepsilon)-\ln s(\pi,\pi^\circ))+\Delta_\varepsilon$,~\eqref{eq:liminf_unified}, and~\eqref{eq:pr:interior:2}, we obtain
\[
    \liminf_{\varepsilon\downarrow0}\frac{\Phi(\pi'_\varepsilon)-\Phi(\pi^\circ)}{\varepsilon}
    \ \ge\ \big(\kappa\,\mathbb{E}_{\widehat{\mathbb{P}}}[Y]-1\big)-C\ >\ 0 .
\]
In particular $\Phi(\pi'_\varepsilon)>\Phi(\pi^\circ)$ for all sufficiently small $\varepsilon>0$, so $\pi^\circ$ is not a maximiser of $\Phi$. As $\pi^\circ$ was an arbitrary boundary point, every maximiser of $\Phi$ lies in $\mathcal{S}_0$.
\end{proof}

As we show now, Assumption~\eqref{A:support} and Assumption~\eqref{A:growth_gamma} guarantee that the solution to the Bellman equation is unique up to a constant.

\begin{theorem}\label{th:uniqueness}
    Assume~\eqref{B:waiting_times},~\eqref{A:integrability_log},~\eqref{A:support}, and~\eqref{A:growth_gamma}. Also, let $\gamma\leq 0$ and let $\bar{v}^\gamma_1\in \mathcal{C}_b(\mathcal{S})$ and $\bar{v}^\gamma_2\in \mathcal{C}_b(\mathcal{S})$ be solutions to~\eqref{eq:Bellman_log} with concave lifts and the same constant $\bar\lambda^\gamma$. Then, there exists a constant $D\in \mathbb{R}$ such that $\bar{v}^\gamma_1\equiv \bar{v}^\gamma_2 + D$.
\end{theorem}
\begin{proof}
    Let us fix some $\bar\pi\in \mathcal{S}$ and define 
    \[
    \bar{v}^\gamma_3(\pi):=\bar{v}^\gamma_1(\pi) - \bar{v}^\gamma_1(\bar\pi) \quad \text{and}\quad \bar{v}^\gamma_4(\pi):=\bar{v}^\gamma_2(\pi) - \bar{v}^\gamma_2(\bar\pi), \quad \pi \in \mathcal{S}.
    \]
    Then, $\bar{v}^\gamma_3$ and $\bar{v}^\gamma_4$ are also solutions to~\eqref{eq:Bellman_log}. We show that $\bar{v}^\gamma_3\equiv \bar{v}^\gamma_4$ which means that $\bar{v}^\gamma_1$ and $\bar{v}^\gamma_2$ differ by a constant. 

    For the contradiction, suppose that $\bar{v}^\gamma_3\not\equiv \bar{v}^\gamma_4$. Then, possibly after swapping the roles of $\bar{v}^\gamma_3$ and $\bar{v}^\gamma_4$, we may assume that there exists $\hat\pi\in \mathcal{S}$ such that 
    \[
    B:=\inf_{\pi\in \mathcal{S}}(\bar{v}^\gamma_3(\pi)-\bar{v}^\gamma_4(\pi) ) = \bar{v}^\gamma_3(
    \hat\pi)-\bar{v}^\gamma_4(\hat\pi)<0. 
\]
Hence, we have $\bar{v}^\gamma_3(
    \pi)-\bar{v}^\gamma_4(\pi)\geq B$, $\pi \in \mathcal{S}$ and the set $\mathcal{D}:=\{\pi \in \mathcal{S}\colon  \bar{v}^\gamma_3(
    \pi)-\bar{v}^\gamma_4(\pi)> B \}$ is open and non-empty as $\bar\pi\in \mathcal{D}$. Also, let $\tilde\pi\in \mathcal{S}$ be the maximiser for $\bar{v}^\gamma_4(\hat\pi)$. Using Assumption~\eqref{A:growth_gamma} and Proposition~\ref{pr:interior}, we have $\tilde\pi\in \mathcal{S}_0$. Then, using Assumption~\eqref{A:support} we have $\mathbb{P}[G(\tilde\pi, R(\eta_1))\in \mathcal{D}]>0$. Thus, we obtain
\begin{align*}
    \bar{v}^\gamma_3(\hat\pi)&\geq  \ln s(\hat\pi,\tilde\pi)-\bar\lambda^\gamma\mathbb{E}[\eta_1]+\mu^\gamma\left[  \ln\langle \tilde\pi,R(\eta_1)\rangle+\bar{v}^\gamma_3 (G(\tilde\pi,R(\eta_1)))\right] \\
    &> \ln s(\hat\pi,\tilde\pi)-\bar\lambda^\gamma\mathbb{E}[\eta_1]+\mu^\gamma\left[  \ln\langle \tilde\pi,R(\eta_1)\rangle+\bar{v}^\gamma_4 (G(\tilde\pi,R(\eta_1)))\right]+ B= \bar{v}^\gamma_4(\hat\pi)+B.
\end{align*}
Thus, we get the contradiction as
\[
\bar{v}^\gamma_3(\hat\pi) - \bar{v}^\gamma_4(\hat\pi) > B = \bar{v}^\gamma_3(\hat\pi) - \bar{v}^\gamma_4(\hat\pi).
\]
Then, $\bar{v}^\gamma_3\equiv \bar{v}^\gamma_4$ which concludes the proof.
\end{proof}
\begin{remark}
    Assumption~\eqref{A:growth_gamma} is used in the proof of Theorem~\ref{th:uniqueness} only to show that the maximiser in~\eqref{eq:Bellman_log} is in the interior of $\mathcal{S}$. Alternatively, one can assume that if a particular asset is not used in an optimal strategy, then it is removed from the model.
\end{remark}

\section{Vanishing risk-aversion asymptotics}\label{S:asymptotics}

In this section, we show that the optimal value of the risk-sensitive portfolio optimisation problem converges to the optimal value of the risk-neutral problem as $\gamma \to 0$. This shows that a complex risk-sensitive problem could be approximated by its simpler risk-neutral version for $\gamma$ close to $0$. 

 First, note that in Theorem~\ref{th:existence_concave}, for any $\gamma<0$, we showed the existence of a solution $(\bar{v}^\gamma,\bar\lambda^\gamma)\in \mathcal{C}_b(\mathcal{S})\times \mathbb{R}$ to~\eqref{eq:Bellman_log}. Also, from Theorem~\ref{th:verification}, we get that the constant $\bar\lambda^\gamma$ is the optimal value of the portfolio optimisation problem linked to~\eqref{eq:J(pi)_log}. Thus, using the fact that the entropic utility measure is increasing with respect to $\gamma$, we may define
 \begin{equation}\label{eq:lambda^0}
     \bar\lambda^0:=\lim_{\gamma\to 0}\bar\lambda^\gamma.
 \end{equation}
 In the following theorem, based on $(\bar{v}^\gamma,\bar\lambda^\gamma)$, we find a map $\bar{v}^0$ such that the pair $(\bar{v}^0, \bar\lambda^0)$ is a solution to~\eqref{eq:Bellman_log} with $\gamma=0$.

\begin{theorem}\label{th:gamma_asympt}
    Assume~\eqref{B:increments},~\eqref{B:waiting_times}, and~\eqref{A:integrability_log}. Also, let $\bar\lambda^0$ be given by~\eqref{eq:lambda^0}. Then, there exists $\bar{v}^0\in \mathcal{C}_b(\mathcal{S})$ such that $(\bar{v}^0, \bar\lambda^0)$ is a solution to~\eqref{eq:Bellman_log} with $\gamma=0$.

\end{theorem}
\begin{proof}

First, let us recall Theorem~\ref{th:existence_concave} and, for any $\gamma<0$, denote by $(\bar{v}^\gamma,\bar\lambda^\gamma)\in \mathcal{C}_b(\mathcal{S})\times \mathbb{R}$ a solution to~\eqref{eq:Bellman_log}. Next, for a fixed $\bar\pi\in \mathcal{S}$, let us define
\[
\hat{v}^\gamma(\pi):=\bar{v}^\gamma(\pi)-\bar{v}^\gamma(\bar\pi), \quad \pi\in \mathcal{S}.
\]
Then, by a direct calculation, we note that $(\hat{v}^\gamma, \bar\lambda^\gamma)$ is also a solution to~\eqref{eq:Bellman_log}. 

Now, using Lemma~\ref{lm:apriori}, we get that the family of functions $(\hat{v}^\gamma)_{\gamma<0}$ is uniformly bounded and equicontinuous. Thus, using the Arzel\`{a}-Ascoli, we may choose a sequence of functions $(\hat{v}^{\gamma_k})$ with $\gamma_k\uparrow 0$ such that $\hat{v}^{\gamma_k}$ converges uniformly to some $\bar{v}^0\in \mathcal{C}_b(\mathcal{S})$. Now, we show that $(\bar{v}^0, \bar\lambda^0)$ is a solution to~\eqref{eq:Bellman_log} with $\gamma=0$. Noting that
\[
    \hat{v}^{\gamma_k}(\pi)+\bar\lambda^{\gamma_k}\mathbb{E}[\eta_1]=\sup_{\pi'\in \mathcal{S}}\mu^{\gamma_k}\left[ \ln s(\pi,\pi')+ \ln\langle\pi',R(\eta_1)\rangle+\hat{v}^{\gamma_k} (G(\pi',R(\eta_1)))\right], \quad \pi\in \mathcal{S},
\]
and using the fact that the left-hand side converges to $\bar{v}^{0}(\pi)+\bar\lambda^{0}\mathbb{E}[\eta_1]$ as $k\to\infty$, it is enough to show that for any $\pi\in \mathcal{S}$, we have
\begin{multline*}
    D_k:=\left|\sup_{\pi'\in \mathcal{S}}\mu^{\gamma_k}\left[ \ln s(\pi,\pi')+ \ln\langle\pi',R(\eta_1)\rangle+\hat{v}^{\gamma_k} (G(\pi',R(\eta_1)))\right] \right.\\
    \left.- \sup_{\pi'\in \mathcal{S}}\mu^{0}\left[ \ln s(\pi,\pi')+ \ln\langle\pi',R(\eta_1)\rangle+\bar{v}^{0} (G(\pi',R(\eta_1)))\right] \right|\to 0, \quad k\to \infty.
\end{multline*}
Note that we have
\begin{multline*}
    D_k\leq \sup_{\pi'\in \mathcal{S}}\left|\mu^{\gamma_k}\left[  \ln\langle\pi',R(\eta_1)\rangle+\hat{v}^{\gamma_k} (G(\pi',R(\eta_1)))\right] \right.\\
    \left.- \mu^{0}\left[ \ln\langle\pi',R(\eta_1)\rangle+\bar{v}^{0} (G(\pi',R(\eta_1)))\right] \right|.
\end{multline*}
Thus, to conclude the proof, it is enough to show
\begin{multline}\label{eq:th:4:1}
    \lim_{k\to\infty} \sup_{\pi'\in \mathcal{S}}|\mu^{\gamma_k}\left[ \ln\langle \pi',R(\eta_1)\rangle+\hat{v}^{\gamma_k} (G(\pi',R(\eta_1)))\right]\\
    -\mu^{\gamma_k}\left[  \ln \langle\pi',R(\eta_1)\rangle+\bar{v}^{0} (G(\pi',R(\eta_1)))\right]|=0
\end{multline}
as well as
\begin{multline}\label{eq:th:4:10}
    \lim_{k\to\infty} \sup_{\pi'\in \mathcal{S}}|\mu^{\gamma_k}\left[  \ln\langle\pi',R(\eta_1)\rangle+\bar{v}^{0} (G(\pi',R(\eta_1)))\right]\\
    -\mathbb{E}\left[ \ln\langle\pi',R(\eta_1)\rangle+\bar{v}^{0} (G(\pi',R(\eta_1)))\right]|=0.
\end{multline}
To see~\eqref{eq:th:4:1}, it is enough to use~\eqref{eq:mu_basic} to get
\begin{multline*}
    \lim_{k\to\infty} \sup_{\pi'\in \mathcal{S}}|\mu^{\gamma_k}\left[  \ln\langle\pi',R(\eta_1)\rangle+\hat{v}^{\gamma_k} (G(\pi',R(\eta_1)))\right]\\
    -\mu^{\gamma_k}\left[\ln\langle\pi',R(\eta_1)\rangle+\bar{v}^{0} (G(\pi',R(\eta_1)))\right]|\leq  \lim_{k\to\infty} \Vert \hat{v}^{\gamma_k}-\bar{v}^0\Vert =0.
\end{multline*}
Next, to get~\eqref{eq:th:4:10}, we recall that $\mu^\gamma[\cdot]$ converges increasingly to $ \mathbb{E}[\cdot]$ as $\gamma\uparrow 0$. Also, using Lemma~\ref{lm:Feller}, we get that the maps
\begin{align*}
    H^k\colon \mathcal{S}\ni\pi'&\to \mu^{\gamma_k}\left[\ln \langle\pi',R(\eta_1)\rangle+ \bar{v}^0(G(\pi',R(\eta_1)))\right],\\
    H^0\colon \mathcal{S}\ni\pi'&\to \mathbb{E}\left[\ln \langle\pi',R(\eta_1)\rangle  + \bar{v}^0(G(\pi',R(\eta_1)))\right]
\end{align*}
are continuous. Thus, using Dini's theorem, we conclude that $H^k$ converges to $H^0$
uniformly with respect to $\pi'\in \mathcal{S}$. This shows~\eqref{eq:th:4:10} and concludes the proof.
\end{proof}

Combining Theorem~\ref{th:gamma_asympt} with Theorem~\ref{th:verification}, we can see that the optimal value of the underlying risk-sensitive optimisation problem converges to the optimal value of the risk-neutral problems as the risk-aversion coefficient converges to $0$. 

\begin{corollary}\label{cor:gamma_asympt_value}
Under the assumptions of Theorem~\ref{th:gamma_asympt}, we have
\begin{align*}
    \sup_{\mathbf\Pi\in \mathcal{A}(\pi)}J^0(\mathbf\Pi,\pi)&=\bar\lambda^0=\lim_{\gamma\to 0}\bar\lambda^\gamma = \lim_{\gamma\to 0}\sup_{\mathbf\Pi\in \mathcal{A}(\pi)}J^\gamma(\mathbf\Pi,\pi).
\end{align*}

\end{corollary}

If we additionally assume~\eqref{A:support} and~\eqref{A:growth_gamma}, we can show that $\bar{v}^0$ is a uniform limit of $\bar{v}^\gamma$.

\begin{corollary}\label{cor:uniform_convergence}
    Assume~\eqref{B:increments},~\eqref{B:waiting_times},~\eqref{A:integrability_log},~\eqref{A:support}, and~\eqref{A:growth_gamma}. Also for any $\gamma\leq 0$, let $(\bar{v}^\gamma, \bar\lambda^\gamma)\in\mathcal{C}_b(\mathcal{S})\times \mathbb{R}$ be the solutions to~\eqref{eq:Bellman_log} with concave lifts and such that $\bar{v}^\gamma(\bar\pi)=0$ for a fixed $\bar\pi \in \mathcal{S}$. Then, we have
    \[
    \lim_{\gamma\uparrow 0}\Vert \bar{v}^\gamma - \bar{v}^0\Vert =0.
    \]
\end{corollary}
\begin{proof}
    First, note that by Theorem~\ref{th:uniqueness}, the maps $ \bar{v}^\gamma$ are uniquely defined. Also, from the proof of Theorem~\ref{th:gamma_asympt} we conclude that from the sequence $ \bar{v}^{\gamma_k}$ we may choose a uniformly convergent subsequence to $\bar{v}^0$. As the limit is also uniquely defined, the convergence is uniform along the original sequence, which concludes the proof.
\end{proof}

Now, we show that a strategy which is optimal for the risk-neutral problem is also almost optimal for the risk-sensitive problem. Note that the argument does not use~\eqref{A:support} nor~\eqref{A:growth_gamma}. 

\begin{theorem}\label{th:strategy_gamma_asympt}
 Assume~\eqref{B:increments},~\eqref{B:waiting_times}, and~\eqref{A:integrability_log}. Let $\hat{\mathbf\Pi}^0=(\hat{\pi}^0(t))_{t\geq 0}$ be an optimal strategy for $J^0$ induced by the Bellman equation~\eqref{eq:Bellman_log} with $\gamma=0$. Then, we have
    \begin{equation}
        \lim_{\gamma\uparrow 0}\sup_{\pi \in \mathcal{S}}|J^\gamma(\hat{\mathbf\Pi}^0,\pi)-\sup_{\mathbf\Pi\in \mathcal{A}(\pi)}J^\gamma(\mathbf\Pi,\pi)|=0.
    \end{equation}
\end{theorem}
\begin{proof}
    Recall that by Theorem~\ref{th:verification}, for any $\pi \in \mathcal{S}$, we have $\sup_{\mathbf\Pi\in \mathcal{A}(\pi)}J^\gamma(\mathbf\Pi,\pi)=\bar\lambda^\gamma$, where $\bar\lambda^\gamma$ is from Equation~\eqref{eq:Bellman_log}. Also, using Theorem~\ref{th:gamma_asympt} and Theorem~\ref{th:verification}, we get 
    \[
    \lim_{\gamma\uparrow 0}\bar\lambda^\gamma = \bar\lambda^0 = \sup_{\mathbf\Pi\in \mathcal{A}(\pi)}J^0(\mathbf\Pi,\pi) = J^0(\hat{\mathbf\Pi}^0,\pi).
    \]
    Thus, noting that
\begin{multline*}
        \sup_{\pi \in \mathcal{S}}|J^\gamma(\hat{\mathbf\Pi}^0,\pi)-\sup_{\mathbf\Pi\in \mathcal{A}(\pi)}J^\gamma(\mathbf\Pi,\pi)| = \sup_{\pi \in \mathcal{S}}|J^\gamma(\mathbf{\hat\Pi}^0,\pi)-\bar\lambda^\gamma|\\
        \leq \sup_{\pi \in \mathcal{S}}|J^\gamma(\mathbf{\hat\Pi}^0,\pi)-\bar\lambda^0|+|\bar\lambda^\gamma - \bar\lambda^0|,
\end{multline*}
    to conclude the proof, it is enough to show
\begin{equation}\label{eq:th:strategy_gamma_asympt:1}
        \lim_{\gamma\uparrow 0}\sup_{\pi \in \mathcal{S}}|J^\gamma(\hat{\mathbf\Pi}^0,\pi)-\bar\lambda^0|=0.
\end{equation}

    We start by noting that, by Lemma~\ref{lm:Feller}, for any $\gamma<0$, the maps
    \begin{align*}
        \mathcal{S}\ni &\pi\to \mathbb{E}\left[\ln \langle\pi',R(\eta_1)\rangle + \bar{v}^0(G(\pi',R(\eta_1))) \right], \\
        \mathcal{S}\ni &\pi\to \mu^\gamma\left[\ln \langle\pi',R(\eta_1)\rangle + \bar{v}^0(G(\pi',R(\eta_1))) \right]
    \end{align*}
    are continuous. Also, since $\mu^\gamma[\cdot]$ converges increasingly to $\mathbb{E}[\cdot]$ as $\gamma\uparrow 0$, by Dini's theorem, for any $\gamma<0$, there exists $d^\gamma\in \mathbb{R}$ such that
\begin{multline}\label{eq:th:strategy_gamma_asympt:2}
        \sup_{\pi'\in \mathcal{S}}\left|\mathbb{E}\left[\ln \langle\pi',R(\eta_1)\rangle + \bar{v}^0(G(\pi',R(\eta_1))) \right] \right. \\
        \left.- \mu^\gamma\left[\ln \langle\pi',R(\eta_1)\rangle + \bar{v}^0(G(\pi',R(\eta_1))) \right] \right|\leq d^\gamma,
    \end{multline}
    and $\lim_{\gamma\uparrow 0}d^\gamma=0$. Next, using Equation~\eqref{eq:Bellman_log} with $\gamma=0$, for any $\pi=\pi(0-)\in \mathcal{S}$, we  obtain
    \begin{align}\label{eq:th:strategy_gamma_asympt:3}
        \bar{v}^0(\pi(0-))+\bar\lambda^0 \mathbb{E}[\eta_1] &= \ln s(\pi(0-),\hat{\pi}^0(0))\nonumber\\
        & \phantom{=}+\mathbb{E}\left[\ln \langle\hat{\pi}^0(0),R(\eta_1)\rangle +\bar{v}^0(G(\hat{\pi}^0(0),R(\eta_1))) \right] \nonumber\\
        & \leq \ln s(\pi(0-),\hat{\pi}^0(0))\nonumber\\
        &\phantom{=}+\mu^\gamma\left[\ln \langle\hat{\pi}^0(0),R(\eta_1)\rangle + \bar{v}^0(G(\hat{\pi}^0(0),R(\eta_1))) \right]+ d^\gamma.
    \end{align}
    Also, recalling that $\pi(\tau_1^-) = G(\hat{\pi}^0(0),R(\eta_1))$ and using Equation~\eqref{eq:Bellman_log} again, by the law of total expectation for the entropic utility measure, we get
    \begin{align*}
        \mu^\gamma&\left[\ln \langle\hat{\pi}^0(0),R(\tau_1)\rangle + \bar{v}^0(\pi(\tau_1^-)) \right] +\bar\lambda^0  \mathbb{E}[\eta_1]\\
        &= \mu^\gamma\left[\ln \langle\hat{\pi}^0(0),R(\tau_1)\rangle  \right.\\
        &\phantom{=}\left.+ \mathbb{E}\left[\ln s(\pi(\tau_1^-),\hat{\pi}^0(\tau_1))+\ln \langle\hat{\pi}^0(\tau_1),R(\tau_1)\rangle+\bar{v}^0(G(\hat{\pi}^0(\tau_1),R(\tau_1))) |\pi(\tau_1-)\right]  \right]\\
        &\leq \mu^\gamma\left[\ln \langle\hat{\pi}^0(0),R(\tau_1)\rangle+\ln s(\pi(\tau_1^-),\hat{\pi}^0(\tau_1))\right.\\
        &\phantom{=}\left.+ \mu^\gamma\left[\ln \langle\hat{\pi}^0(\tau_1),R(\tau_1)\rangle+\bar{v}^0(G(\hat{\pi}^0(\tau_1),R(\tau_1))) |\pi(\tau_1^-)\right]  \right] + d^\gamma\\
        & = \mu^\gamma\left[\ln \langle\hat{\pi}^0(0),R(\tau_1)\rangle 
        + \ln s(\pi(\tau_1-),\hat{\pi}^0(\tau_1))\right.\\
        &\phantom{=}\left.+\ln \langle\hat{\pi}^0(\tau_1),R(\tau_1,\tau_2)\rangle+\bar{v}^0(G(\hat{\pi}^0(\tau_1),R(\tau_1,\tau_2))) \right]  + d^\gamma\\
        & = \mu^\gamma\left[\ln \langle\hat{\pi}^0(0),R(\tau_1)\rangle  
        + \ln s(\pi(\tau_1-),\hat{\pi}^0(\tau_1))\right.\\
        &\phantom{=}\left.+\ln \langle\hat{\pi}^0(\tau_1),R(\tau_1,\tau_2)\rangle+\bar{v}^0(\pi(\tau_2-)) \right]  + d^\gamma.
    \end{align*}
    Consequently, using~\eqref{eq:th:strategy_gamma_asympt:3}, we have
    \begin{multline*}
        \bar{v}^0(\pi(0-))+2\bar\lambda^0 \mathbb{E}[\eta_1] \\
        \leq  \mu^\gamma\left[\sum_{i=0}^1 \left( \ln s(\pi(\tau_i^-),\hat{\pi}^0(\tau_i)) +\ln \langle\pi(\tau_i ),R(\tau_i,\tau_{i+1})\rangle \right) +\bar{v}^0(\pi(\tau_2-))\right]+2d^\gamma.
    \end{multline*}
    Repeating this argument, for any $n\geq 2$, we obtain
        \begin{multline*}
        \bar{v}^0(\pi(0-))+n\bar\lambda^0  \mathbb{E}[\eta_1]\\
        \leq  \mu^\gamma\left[\sum_{i=0}^{n-1} \left( \ln s(\pi(\tau_i^-),\hat{\pi}^0(\tau_i)) +\ln \langle\hat{\pi}^0(\tau_i ),R(\tau_i,\tau_{i+1})\rangle \right)  +\bar{v}^0(\pi(\tau_n^-))\right]+2nd^\gamma.
    \end{multline*}
    Thus, dividing both sides by $n \mathbb{E}[\eta_1]=\mathbb{E}[\tau_n]$ and using the fact that $\bar{v}^0$ is bounded, we get
    \[
    \bar\lambda^0 \leq J^\gamma(\hat{\mathbf\Pi}^0, \pi) + 2d^\gamma.
    \]
    Consequently, we obtain
    \[
    \bar\lambda^0 - 2d^\gamma \leq J^\gamma(\hat{\mathbf\Pi}^0, \pi) \leq \bar\lambda^0.
    \]
    This combined with~\eqref{eq:th:strategy_gamma_asympt:1} and $\lim_{\gamma\uparrow 0} d^\gamma=0$, concludes the proof.
\end{proof}

Now, we complement the convergence of optimal values with the convergence of optimal strategies. To get the result, we need to amplify Assumption~\eqref{A:integrability_log} to the following condition.
\begin{enumerate}
\item[(\namedlabel{A:integrability_log_2}{$\mathcal{A}3'$})] (Integrability).  For any $\gamma<0$, we have
\[
\mathbb{E}\left[e^{\gamma  \min_{i=1, \ldots, d}r_i(\eta_1)}\right]<\infty ,\qquad\text{and}\qquad \mathbb{E}\Big[(\max_{i=1,\ldots,d}r_i(\eta_1))^2\Big]<\infty .
\]
\end{enumerate}
It should be noted that Assumption~\eqref{A:integrability_log_2} simply adds square integrability of $\max\limits_{i=1,\ldots,d}r_i(\eta_1)$.

\begin{theorem}\label{th:gamma_asympt_strategy} Assume~\eqref{B:increments},~\eqref{B:waiting_times},~\eqref{A:integrability_log_2},~\eqref{A:support}, and~\eqref{A:growth_gamma}. Also, let $\pi\in \mathcal{S}$ and, for any $\gamma\leq 0$, let $\pi^\gamma=\pi^\gamma(\pi)$ be the maximiser in~\eqref{eq:Bellman_log}. Then, we have $\pi^\gamma\to \pi^0$ as $\gamma\uparrow 0$.
\end{theorem}
\begin{proof}
First, note that by Theorem~\ref{th:uniqueness} and Theorem~\ref{th:unique_maximiser}, for any $\gamma\leq 0$, the solutions to the Bellman equations are unique (up to a constant) and admit unique maximisers. Let $\gamma_k\uparrow 0$ be any sequence such that $\pi^{\gamma_k}\to\bar\pi$ as $k\to \infty$ for some $\bar\pi \in \mathcal{S}$; such a sequence exists by the compactness of $\mathcal{S}$. Since the maximiser for $\gamma=0$ is unique, see Theorem~\ref{th:unique_maximiser}, it is enough to show that $\bar\pi=\pi^0$, i.e.\ that
\begin{equation}\label{eq:th:gamma_asympt:1}
        \bar{v}^0(\pi)+\bar\lambda^0 \mathbb{E}[\eta_1]= \ln s(\pi,\bar\pi)+\mathbb{E}\left[\ln \langle\bar\pi,R(\eta_1)\rangle+\bar{v}^0(G(\bar\pi,R(\eta_1)))\right].
    \end{equation}
    Indeed, then $\bar\pi$ is a maximiser in~\eqref{eq:Bellman_log} for $\gamma=0$, hence $\bar\pi=\pi^0$ by uniqueness; and, as every convergent subsequence of $(\pi^{\gamma_k})$ has the same limit $\pi^0$ and $\mathcal{S}$ is compact, this yields $\pi^\gamma\to \pi^0$ as $\gamma\uparrow 0$. 

    By definition, we have
    \begin{equation}\label{eq:th:gamma_asympt:2}
        \bar{v}^{\gamma_k}(\pi)+\bar\lambda^{\gamma_k}\mathbb{E}[\eta_1] = \ln s(\pi,\pi^{\gamma_k})+\mu^{\gamma_k}\left[\ln \langle\pi^{\gamma_k},R(\eta_1)\rangle+\bar{v}^{\gamma_k}(G(\pi^{\gamma_k},R(\eta_1)))\right].
    \end{equation}
    From the proof of Theorem~\ref{th:gamma_asympt}, by passing to a subsequence if needed, it can be assumed that $\bar{v}^{\gamma_k}$ converges uniformly to $\bar{v}^0$ as $k\to\infty$. Also, using Corollary~\ref{cor:gamma_asympt_value} and the continuity of $s$, we have $\bar\lambda^{\gamma_k}\to \bar\lambda^0$ and $\ln s(\pi,\pi^{\gamma_k})\to \ln s(\pi,\bar\pi)$ as $k\to \infty$. Next, to ease the notation, for any $\gamma<0$, $\pi'\in \mathcal{S}$, and $h\in \mathcal{C}_b(\mathcal{S})$, we set
    \begin{align*}
        F^\gamma(\pi',h)&:=\mu^{\gamma}\left[\ln \langle\pi',R(\eta_1)\rangle+h(G(\pi',R(\eta_1)))\right],\\
        F^0(\pi',h)&:=\mathbb{E}\left[\ln \langle\pi',R(\eta_1)\rangle+h(G(\pi',R(\eta_1)))\right].
    \end{align*}
    Thus, in view of~\eqref{eq:th:gamma_asympt:2}, to prove~\eqref{eq:th:gamma_asympt:1} it is enough to show
\begin{equation}\label{eq:th:gamma_asympt:3}
|F^{\gamma_k}(\pi^{\gamma_k},\bar{v}^{\gamma_k}) - F^{0}(\bar\pi,\bar{v}^{0})|\to 0, \quad k\to \infty.
\end{equation}
Now, we have
\begin{align}\label{eq:th:gamma_asympt:split}
    |F^{\gamma_k}(\pi^{\gamma_k},\bar{v}^{\gamma_k}) - F^{0}(\bar\pi,\bar{v}^{0})|&\leq |F^{\gamma_k}(\pi^{\gamma_k},\bar{v}^{\gamma_k}) - F^{\gamma_k}(\pi^{\gamma_k},\bar{v}^{0})| \nonumber\\
    &\phantom{\leq}+ |F^{\gamma_k}(\pi^{\gamma_k},\bar{v}^{0}) -F^{\gamma_k}(\bar\pi,\bar{v}^{0}) |\\
    & \phantom{\leq}+ |F^{\gamma_k}(\bar\pi,\bar{v}^{0}) - F^{0}(\bar\pi,\bar{v}^{0})|,\nonumber
\end{align}
and we show that each of the three terms on the right-hand side tends to $0$ as $k\to \infty$. For transparency, we split remaining argument into three steps: (1) proof that $|F^{\gamma_k}(\pi^{\gamma_k},\bar{v}^{\gamma_k}) - F^{\gamma_k}(\pi^{\gamma_k},\bar{v}^{0})|\to 0$ as $k\to\infty$; (2) proof that $|F^{\gamma_k}(\pi^{\gamma_k},\bar{v}^{0}) -F^{\gamma_k}(\bar\pi,\bar{v}^{0}) |\to 0$ as $k\to\infty$; (3) proof that $|F^{\gamma_k}(\bar\pi,\bar{v}^{0}) - F^{0}(\bar\pi,\bar{v}^{0})|\to 0$ as $k\to\infty$.

\medskip
\textsc{Step 1.} We show that $|F^{\gamma_k}(\pi^{\gamma_k},\bar{v}^{\gamma_k}) - F^{\gamma_k}(\pi^{\gamma_k},\bar{v}^{0})|\to 0$ as $k\to\infty$. Recalling that $\bar{v}^{\gamma_k}$ converges uniformly to $\bar{v}^0$ and using the translation invariance and the monotonicity of $\mu^{\gamma_k}$, see~\eqref{eq:mu_basic}, we obtain
\[
|F^{\gamma_k}(\pi^{\gamma_k},\bar{v}^{\gamma_k}) - F^{\gamma_k}(\pi^{\gamma_k},\bar{v}^{0})|\leq \Vert\bar{v}^{\gamma_k}-\bar{v}^{0}\Vert \to 0, \quad k\to \infty,
\]
which concludes the proof of this step.

\medskip
\textsc{Step 2.} We show that $|F^{\gamma_k}(\pi^{\gamma_k},\bar{v}^{0}) -F^{\gamma_k}(\bar\pi,\bar{v}^{0}) |\to 0$ as $k\to\infty$. Before proceeding, let us record an integrability bound used repeatedly. Since $R_i(\eta_1)=e^{r_i(\eta_1)}$, for any $\pi'\in \mathcal{S}$ we have $\langle\pi',R(\eta_1)\rangle\in [e^{\min_{i}r_i(\eta_1)},e^{\max_{i}r_i(\eta_1)}]$, so that, for any $h\in \mathcal{C}_b(\mathcal{S})$ with $\Vert h\Vert\leq \ln\tfrac{1}{\bar s}$, we obtain
\begin{equation}\label{eq:th:gamma_asympt:envelope}
    \left|\ln \langle\pi',R(\eta_1)\rangle+h(G(\pi',R(\eta_1)))\right|\leq H, \qquad
    H:=|\min_{i}r_i(\eta_1)|+|\max_{i}r_i(\eta_1)|+\ln\tfrac{1}{\bar s}.
\end{equation}
Also, by Assumption~\eqref{A:integrability_log_2} and Cauchy inequality, we have $\mathbb{E}[H]<\infty$ and, for any $\gamma<0$, we get
\begin{equation}\label{eq:th:gamma_asympt:Hexp}
    \mathbb{E}\left[H\,e^{\gamma \min_{i}r_i(\eta_1)}\right]<\infty.
\end{equation}

Now, for any $k\in \mathbb{N}$, let us define
\begin{align*}
    X_k&:=\ln \langle\pi^{\gamma_k},R(\eta_1)\rangle+\bar{v}^0(G(\pi^{\gamma_k},R(\eta_1))), \\
    Y&:=\ln \langle\bar\pi,R(\eta_1)\rangle+\bar{v}^0(G(\bar\pi,R(\eta_1))),\\
    \Delta_k&:=X_k-Y,
\end{align*}
so that the second term in~\eqref{eq:th:gamma_asympt:split} equals $|\mu^{\gamma_k}[X_k]-\mu^{\gamma_k}[Y]|$. In the following, we use two properties of $\Delta_k$. First, since $\pi^{\gamma_k}\to\bar\pi$ and, for each fixed $R\in (0,\infty)^d$, the maps $\pi'\to \ln \langle\pi',R\rangle$ and $\pi'\to \bar{v}^0(G(\pi',R))$ are continuous, we have
\begin{equation}\label{eq:th:gamma_asympt:as}
    \Delta_k\to 0, \quad \mathbb{P}\text{-a.s.}, \quad k\to \infty.
\end{equation}
Second, by~\eqref{eq:th:gamma_asympt:envelope} applied with $h=\bar{v}^0$, both $X_k$ and $Y$ are bounded in absolute value by the integrable random variable $H$; in particular,
\begin{equation}\label{eq:th:gamma_asympt:dom}
    X_k\geq -H, \qquad Y\geq -H, \qquad |\Delta_k|\leq 2H.
\end{equation}

Next, using the robust representation of the entropic utility, for $\gamma<0$ and any random variable $U$ with $\mathbb{E}[e^{\gamma U}]<\infty$, we have
\begin{equation}\label{eq:th:gamma_asympt:DV}
    \mu^{\gamma}[U]=\frac{1}{\gamma}\ln \mathbb{E}\!\left[e^{\gamma U}\right]=\inf_{Q\ll \mathbb{P}}\left(\mathbb{E}_{Q}[U]-\frac{1}{\gamma}H(Q\,|\,\mathbb{P})\right),
\end{equation}
where $H(Q\,|\,\mathbb{P})$ denotes the relative entropy of $Q$ with respect to $\mathbb{P}$, and the infimum is attained at the measure $Q^\gamma_U$ given by $\tfrac{dQ^\gamma_U}{d\mathbb{P}}:=\tfrac{e^{\gamma U}}{\mathbb{E}[e^{\gamma U}]}$; see, e.g.,~\cite[Example~4.34]{FolSch2016} or~\cite[Lemma A.1]{Bar2019} for details. Thus, using $Q^{\gamma_k}_Y$ as a (sub-optimal) test measure in the representation~\eqref{eq:th:gamma_asympt:DV} of $\mu^{\gamma_k}[X_k]$, we obtain
\[
\mu^{\gamma_k}[X_k]-\mu^{\gamma_k}[Y]\leq \mathbb{E}_{Q^{\gamma_k}_Y}[X_k]-\mathbb{E}_{Q^{\gamma_k}_Y}[Y]=\mathbb{E}_{Q^{\gamma_k}_Y}[\Delta_k]\leq \mathbb{E}_{Q^{\gamma_k}_Y}[|\Delta_k|].
\]
By symmetry, interchanging the roles of $X_k$ and $Y$, we also get $\mu^{\gamma_k}[Y]-\mu^{\gamma_k}[X_k]\leq \mathbb{E}_{Q^{\gamma_k}_{X_k}}[|\Delta_k|]$, which implies
\begin{equation}\label{eq:th:gamma_asympt:5}
    |\mu^{\gamma_k}[X_k]-\mu^{\gamma_k}[Y]|\leq \max\left(\frac{\mathbb{E}\!\left[|\Delta_k|\,e^{\gamma_k Y}\right]}{\mathbb{E}[e^{\gamma_k Y}]},  \frac{\mathbb{E}\!\left[|\Delta_k|\,e^{\gamma_k X_k}\right]}{\mathbb{E}[e^{\gamma_k X_k}]}\right).
\end{equation}
It remains to show that both ratios in~\eqref{eq:th:gamma_asympt:5} tend to $0$. We show the argument only for the one involving $Y$; the argument for $X_k$ is analogous. Fix $\gamma^*<0$ with $\gamma_k\in [\gamma^*,0)$ for all $k$ (possible as $\gamma_k\uparrow 0$). Using $Y\geq \min_i r_i(\eta_1)-\ln\tfrac1{\bar s}$ and $\gamma_k<0$, we get the pointwise bound $e^{\gamma_k Y}\leq \bar s^{\,\gamma^*}\,e^{\gamma^* \min_i r_i(\eta_1)}\vee 1$. Thus, we get by~\eqref{eq:th:gamma_asympt:dom},
\[
|\Delta_k|\,e^{\gamma_k Y}\leq 2H\left(1+\bar s^{\,\gamma^*}e^{\gamma^* \min_i r_i(\eta_1)}\right)=:\bar H,
\]
and, by~\eqref{eq:th:gamma_asympt:Hexp}, we have $\mathbb{E}[\bar H]<\infty$. Since $|\Delta_k|\to 0$ $\mathbb{P}$-a.s.\ by~\eqref{eq:th:gamma_asympt:as} and $\gamma_k\uparrow 0$ gives $e^{\gamma_k Y}\to 1$ $\mathbb{P}$-a.s., the dominated convergence theorem yields $\mathbb{E}[|\Delta_k|\,e^{\gamma_k Y}]\to 0$. Also, we obtain $\mathbb{E}[e^{\gamma_k Y}]\to 1$ by the same argument. Hence, we have $\frac{\mathbb{E}\!\left[|\Delta_k|\,e^{\gamma_k Y}\right]}{\mathbb{E}[e^{\gamma_k Y}]}\to 0$ as $k\to \infty$, which concludes the proof of this step.

\medskip
\textsc{Step 3.} We show that $|F^{\gamma_k}(\bar\pi,\bar{v}^{0}) - F^{0}(\bar\pi,\bar{v}^{0})|\to 0$ as $k\to\infty$. In fact, this follows directly from the fact that $\mu^{\gamma_k}[Z]\to \mathbb{E}[Z]$, where we set $Z:=\ln \langle\bar\pi,R(\eta_1)\rangle+\bar{v}^0(G(\bar\pi,R(\eta_1)))$. This concludes the proof.
\end{proof}
\begin{remark}\label{rm:remark_uniqueness}
    In fact, in the proof of Theorem~\ref{th:gamma_asympt_strategy}, we only need Assumption~\eqref{A:growth_gamma} to be satisfied for $\gamma=0$. Indeed, the key element is the uniqueness (up to a constant) of a map $\bar{v}^0$ solving~\eqref{eq:th:gamma_asympt:1} which is guaranteed by Theorem~\ref{th:uniqueness} with $\gamma=0$.
\end{remark}
\section{Discrete time grid approximation}\label{S:discrete}

In this section, we discuss how to approximate the optimal value of~\eqref{eq:J(pi)_log}     with the help of intervention times restricted to the discrete time grid. This problem is important from a numerical perspective, as it requires the simulation of the controlled process trajectory only in the fixed set of grid points.

We start by introducing the auxiliary notation. Let us recall the sequence of waiting times $\eta_i$ from Section~\ref{S:preliminaries} and, for any $m\in \mathbb{N}$, let us define
\begin{equation}\label{eq:eta_m}
    \eta_i^m:=\sum_{j=0}^\infty \frac{j+1}{2^m} 1_{\{\frac{j}{2^m}< \eta_i\leq\frac{j+1}{2^m} \}};
\end{equation}
note that $\eta_i^m$ is simply a non-anticipative approximation of $\eta_i$ with values restricted to the time-grid $\{\frac{k}{2^m}\colon k\in \mathbb{N}_0\}$. Next, we define recursively
\[
\tau_0^m:=0, \quad \tau_{n+1}^m=\tau_n^m+\eta_{n+1}^m, \quad n\in \mathbb{N}.
\]
Also, by $\mathcal{A}_m(\pi)$, we denote the family of self-financing portfolio weights $\Pi=(\pi(t))_{t\geq 0}$ with the intervention times given by $(\tau_n^m)_{n\in \mathbb{N}}$. Next, we use the notation
\begin{equation}\label{eq:J(pi)_log:m}
    J^\gamma_m(\mathbf\Pi,\pi): =\liminf_{n\to\infty}\frac{1}{\mathbb{E}[\tau_n^m]}\mu^\gamma \left[\sum_{i=0}^{n-1} \ln \frac{W(\tau_{i+1}^{m-})}{W(\tau_i^{m-})}\right], \quad m\in \mathbb{N}, \pi \in \mathcal{S}, \mathbf\Pi\in \mathcal{A}_m(\pi).
\end{equation}
In the following theorem, we show that the optimal value of $J^\gamma_m$ converges to the optimal value of $J^\gamma$ as $m\to \infty$. Effectively, this means that the portfolio optimisation problem with arbitrary intervention times could be approximated with optimal portfolio value with intervention times on a discrete time grid. To obtain the result, we use a stronger version of Assumption~\eqref{A:integrability_log} given as
\begin{enumerate}
\item[(\namedlabel{A:integrability_log2}{$\mathcal{A}3''$})] (Integrability).  For any $\gamma<0$ there exists $h>0$ such that
\[
\mathbb{E}\left[e^{\gamma  \min_{i=1, \ldots, d}\inf_{t\in [0,h]}r_i(\eta_1+t)}\right]<\infty,\qquad\text{and}\qquad \mathbb{E}\Big[\max_{i=1,\ldots,d}\sup_{t\in [0,h]}r_i(\eta_1+t)\Big]<\infty.
\]
\end{enumerate}

\begin{theorem}\label{th:discrete_convergence}
    Let $\gamma<0$, and let $J^\gamma$ and $J^\gamma_m$ be given by~\eqref{eq:J(pi)_log} and~\eqref{eq:J(pi)_log:m}, respectively.  Also, assume~\eqref{B:increments},~\eqref{B:waiting_times}, and~\eqref{A:integrability_log2}. Then, for any $\pi \in \mathcal{S}$, we have
    \begin{equation}\label{eq:th:discrete_convergence}
            \lim_{m\to\infty}\sup_{\Pi \in \mathcal{A}_m(\pi)}J^\gamma_m(\Pi,\pi) = \sup_{\Pi \in \mathcal{A}(\pi)}J^\gamma(\Pi,\pi).
    \end{equation}
\end{theorem}

\begin{proof}
    First, note that from Theorem~\ref{th:existence_concave} applied to the waiting times $(\eta_i^m)$, for a sufficiently big $m$ there exists a solution $(\bar{v}^\gamma_m, \bar\lambda^\gamma_m)\in \mathcal{C}_b(\mathcal{S})\times \mathbb{R}$ such that, for any $\pi \in \mathcal{S}$, we have
\begin{equation}\label{eq:Bellman_discrete}
    \bar{v}^\gamma_m(\pi)=\sup_{\pi'\in \mathcal{S}}\left( \ln s(\pi,\pi')-\bar\lambda^\gamma_m\mathbb{E}[\eta_1^m]+\mu^\gamma\left[  \ln\langle\pi',R(\eta_1^m)\rangle+\bar{v}^\gamma_m (G(\pi',R(\eta_1^m)))\right]\right).
\end{equation}
Also, using the argument from Theorem~\ref{th:verification}, we have
\[
\sup_{\Pi\in \mathcal{A}_m(\pi)}J^\gamma_m(\Pi,\pi) = \bar\lambda^\gamma_m, \quad \pi\in \mathcal{S}.
\]
Next, let $\bar\lambda^\gamma_{m_k}$ be subsequence of $\bar\lambda^\gamma_{m}$ which converges to some $\hat\lambda^\gamma$ as $k\to\infty$. As we show later, $\hat\lambda^\gamma$ is equal to $\sup_{\Pi \in \mathcal{A}(\pi)}J^\gamma(\Pi,\pi)$. Then, as the subsequence was chosen arbitrarily, we get~\eqref{eq:th:discrete_convergence}. In fact, to conclude the proof, we find a map $\hat{v}^\gamma\in \mathcal{C}_b(\mathcal{S})$ such that the following equation is satisfied
\begin{equation}\label{eq:Bellman2}
    \hat{v}^\gamma(\pi)=\sup_{\pi'\in \mathcal{S}}\left( \ln s(\pi,\pi')-\hat\lambda^\gamma\mathbb{E}[\eta_1]+\mu^\gamma\left[  \ln\langle\pi',R(\eta_1)\rangle+\hat{v}^\gamma (G(\pi',R(\eta_1)))\right]\right), \quad \pi\in \mathcal{S}.
\end{equation}
Then, using Theorem~\ref{th:verification}, we get $\hat\lambda^\gamma=\sup_{\Pi \in \mathcal{A}(\pi)}J^\gamma(\Pi,\pi)$.

Let $\hat{v}^\gamma_{m_k}(\pi):=\bar{v}^\gamma_{m_k}(\pi)-\bar{v}^\gamma_{m_k}(\bar\pi)$ for some fixed $\bar\pi\in \mathcal{S}$. Then, using Lemma~\ref{lm:apriori} we get that the family of functions $(\hat{v}^\gamma_{m_k})_{k\in \mathbb{N}}$ is uniformly bounded and equicontinous. Thus, by the Arzel\'{a}-Ascoli theorem, we may find a map $\hat{v}^\gamma\in \mathcal{C}_b(\mathcal{S})$ such that  $\hat{v}^\gamma_{m_k}$ converges uniformly along some subsequence (for brevity, denoted by $m_k$) to some $\hat{v}^\gamma \in \mathcal{C}_b(\mathcal{S})$ as $k\to\infty$.

Now, we show that~\eqref{eq:Bellman2} is satisfied. By straightforward calculation, for any $\pi \in \mathcal{S}$ and $k\in \mathbb{N}$, we have
\begin{equation}
    \hat{v}^\gamma_{m_k}(\pi)+\bar\lambda^\gamma_{m_k}\mathbb{E}[\eta_1^{m_k}]=\sup_{\pi'\in \mathcal{S}}\left( \ln s(\pi,\pi')+\mu^\gamma\left[  \ln\langle\pi',R(\eta_1^{m_k})\rangle+\hat{v}^\gamma_{m_k} (G(\pi',R(\eta_1^{m_k})))\right]\right).
\end{equation}
Also, by the construction, we obtain
\[
\lim_{k\to\infty}\left(\hat{v}^\gamma_{m_k}(\pi) + \bar\lambda^\gamma_{m_k}\mathbb{E}[\eta_1^{m_k}]\right) = \hat{v}^\gamma(\pi) + \bar\lambda^\gamma\mathbb{E}[\eta_1].
\]
Thus, setting
\begin{multline*}
    \Delta_k:=\sup_{\pi' \in \mathcal{S}}\left|\mu^\gamma\left[  \ln\langle\pi',R(\eta_1^{m_k})\rangle+\hat{v}^\gamma_{m_k} (G(\pi',R(\eta_1^m)))\right]\right.\\
    \left.- \mu^\gamma\left[  \ln\langle\pi',R(\eta_1)\rangle+\hat{v}^\gamma (G(\pi',R(\eta_1)))\right]\right|
\end{multline*}
and noting that 
\begin{multline*}
    \left|\sup_{\pi'\in \mathcal{S}}\left( \ln s(\pi,\pi')+\mu^\gamma\left[  \ln\langle\pi',R(\eta_1^{m_k})\rangle+\hat{v}^\gamma_{m_k} (G(\pi',R(\eta_1^{m_k})))\right]\right) \right.\\
    \left. -\sup_{\pi'\in \mathcal{S}}\left( \ln s(\pi,\pi')+\mu^\gamma\left[  \ln\langle\pi',R(\eta_1)\rangle+\hat{v}^\gamma (G(\pi',R(\eta_1)))\right]\right) \right|\leq \Delta_k,
\end{multline*}
it is enough to show $\Delta_k\to 0$ as $k\to \infty$. For the contradiction, suppose that there exists a subsequence (for brevity denoted by $\Delta_k$) and some $a>0$ such that $\Delta_k\to a$ as $k\to \infty$. Also, note that using the fact that $\hat{v}^\gamma_{m_k}\in \mathcal{C}_b(\mathcal{S})$ and $\hat{v}^\gamma\in \mathcal{C}_b(\mathcal{S})$, for any $k\in \mathbb{N}$, using Lemma~\ref{lm:Feller}, we may find $\pi'_k\in \mathcal{S}$ such that
\begin{multline}\label{eq:th:discrete_convergence:Delta}
    \Delta_k=\left|\mu^\gamma\left[  \ln\langle\pi'_k,R(\eta_1^{m_k})\rangle+\hat{v}^\gamma_{m_k} (G(\pi'_k,R(\eta_1^{m_k})))\right]\right.\\
    \left.- \mu^\gamma\left[  \ln\langle\pi'_k,R(\eta_1)\rangle+\hat{v}^\gamma (G(\pi'_k,R(\eta_1)))\right]\right|.
\end{multline}
Next, using the compactness of $\mathcal{S}$, we can find a further subsequence of $(\pi'_k)$ (for brevity still denoted by $(\pi'_k)$) and some $\pi'\in \mathcal{S}$ such that $\pi'_k\to \pi'$ as $k\to\infty$. Moreover, using the  right-continuity of paths for R(t), we obtain $\mathbb{P}$-a.s. as $k\to\infty$
\[
\ln\langle\pi'_k,R(\eta_1^{m_k})\rangle \to \ln\langle\pi',R(\eta_1)\rangle \quad \text{and}
\quad \ln\langle\pi'_k,R(\eta_1)\rangle\to \ln\langle\pi',R(\eta_1)\rangle.
\]
Similarly, we have $G(\pi'_k,R(\eta_1^{m_k}))\to G(\pi',R(\eta_1))$ as $k\to \infty$. Then, using the inequality
\begin{multline*}
    \left|\hat{v}^\gamma_{m_k} (G(\pi'_k,R(\eta_1^{m_k})))-\hat{v}^\gamma (G(\pi',R(\eta_1)))\right|\\
    \leq \left|\hat{v}^\gamma_{m_k} (G(\pi'_k,R(\eta_1^{m_k})))-\hat{v}^\gamma (G(\pi'_k,R(\eta_1^{m_k})))\right| \\
    \phantom{=}+\left|\hat{v}^\gamma (G(\pi'_k,R(\eta_1^{m_k})))-\hat{v}^\gamma (G(\pi',R(\eta_1)))\right|,
\end{multline*}
the continuity of $\hat{v}^\gamma$ and the fact that $\hat{v}^\gamma_{m_k}$ converges uniformly to $\hat{v}^\gamma$, we obtain $\hat{v}^\gamma_{m_k} (G(\pi'_k,R(\eta_1^{m_k})))\to\hat{v}^\gamma (G(\pi',R(\eta_1)))$ as $k\to \infty$. Using a similar argument we can show $\hat{v}^\gamma (G(\pi'_k,R(\eta_1)))\to \hat{v}^\gamma (G(\pi',R(\eta_1)))$ as $k\to \infty$. Thus, using Assumption~\eqref{A:integrability_log2}, from~\eqref{eq:th:discrete_convergence:Delta}, we obtain $\Delta_k\to 0$ as $k\to \infty$. This is a contradiction since it was assumed that $\Delta_k\to a>0$. The proof is complete.
\end{proof}


\section{Examples}\label{S:numerical}

In this section, we present a series of examples illustrating our results. First, we discuss exemplary settings where all the main assumptions used in this paper could be explicitly verified. Then, we present numerical simulations confirming the convergence results from Section~\ref{S:asymptotics} and Section~\ref{S:discrete}.

We start by exhibiting concrete market models for which assumptions~\eqref{B:increments}--\eqref{A:growth_gamma} hold.
 
\begin{example}[Heavy positive tails]\label{ex:heavy}
Let $\eta_i\equiv1$ and let the one-period log-returns $r(1)=(r_1(1),\dots,r_d(1))$ have independent components, each with a density positive on $\mathbb{R}$, and such that light left tail (i.e. $\mathbb{E}[e^{\theta (r_i(1))^-}]<\infty$ for every $\theta>0$, as for a Gaussian or exponentially tempered left tail) and a heavy right tail satisfying
\begin{equation}\label{eq:ex:heavy_tail}
    \mathbb{E}\big[(r_i(1))^+\big]<\infty\qquad\text{and}\qquad \mathbb{E}\big[e^{r_i(1)}\big]=\infty ,
\end{equation}
for instance $\mathbb{P}[r_i(1)>t]\sim t^{-2}$ as $t\to\infty$. Then~\eqref{B:increments} and~\eqref{B:waiting_times} hold trivially. Next, note that for any $\gamma\leq 0$, by the light left tail condition, we have
\[
\mathbb{E}[e^{\gamma r_1(1)}] = \mathbb{E}[1_{\{r_1(1)> 0\}}e^{\gamma r_1(1)}+1_{\{r_1(1)\leq 0\}}e^{\gamma r_1(1)}] \leq 1 +\mathbb{E}[e^{|\gamma| (r_1(1))^-}]<\infty.
\]
Thus, using the facts that
\begin{align*}
    \mathbb{E}&[e^{\gamma\min_{i=1, \ldots, d} r_i(1)}] = \mathbb{E}[\max_{i=1, \ldots, d} e^{\gamma r_i(1)}]\leq \sum_{i=1}^d \mathbb{E}[e^{\gamma r_i(1)}]<\infty,\\ 
    \mathbb{E}&[\max_{i=1, \ldots, d} r_i(1)]\le\sum_{i=1}^d\mathbb{E}[(r_i(1))^+]<\infty,
\end{align*}
we get~\eqref{A:integrability_log}. Moreover, the components have full support on $\mathbb{R}^d$, so~\eqref{A:support} holds. Finally, for any $i$ and any $\pi'\in\mathcal{S}$ with $\pi'_i=0$, the independence of the components yields
\[
    \mathbb{E}\big[R_i(1)\,\langle\pi',R(1)\rangle^{\gamma-1}\big]
    =\mathbb{E}[R_i(1)]\;\mathbb{E}\big[\langle\pi',R(1)\rangle^{\gamma-1}\big]=\infty,
\]
because $\mathbb{E}[R_i(1)]=\mathbb{E}[e^{r_i(1)}]=\infty$ by~\eqref{eq:ex:heavy_tail}, while, recalling $\gamma\leq 0$, we have 
\[
\mathbb{E}[\langle\pi',R(1)\rangle^{\gamma-1}]\le \pi_j'^{\,\gamma-1}\,\mathbb{E}[e^{(\gamma-1)r_j(1)}]<\infty
\]
for any $j$ with $\pi'_j>0$, again by the light left tails. Since the right-hand side of~\eqref{eq:growth_gamma} is finite, the resulting infinite numerator at the left-hand side (its denominator being finite by~\eqref{A:integrability_log}) makes~\eqref{A:growth_gamma} hold for every $\gamma\le0$. Consequently, all assumptions~\eqref{B:increments}--\eqref{A:growth_gamma} are satisfied. The same conclusion holds for an exponential-L\'evy model whose L\'evy measure has a heavy positive tail, $\int_{x>1}e^{x}\,\nu_i(dx)=\infty$ together with $\int_{x>1}x\,\nu_i(dx)<\infty$, and a Brownian or exponentially tempered negative part.
\end{example}
 
\begin{example}[Log-normal returns]\label{ex:lognormal}
Let $\ln S$ be a multivariate Brownian motion with drift and let $\eta_i\equiv1$. Then,~\eqref{B:increments} and~\eqref{B:waiting_times} hold trivially. Also, as the Gaussian log-returns have all exponential moments and $R(\eta_1)$ has full support, we have~\eqref{A:integrability_log}  and~\eqref{A:support}. It remains to verify the growth condition~\eqref{A:growth_gamma}. Here, the ratio from~\eqref{eq:growth_gamma} could be bounded explicitly, as we show in the exemplary setting. More specifically, following Remark~\ref{rm:remark_uniqueness}, we focus on $\gamma=0$ and for simplicity we set $d=2$. Then, for a boundary point $\pi'$, the ratio in~\eqref{eq:growth_gamma} is a ratio of log-normal moments. Thus, writing $r_1(\eta_1)-r_2(\eta_1)\sim\mathcal{N}(\mu_\Delta,\sigma_\Delta^2)$, condition~\eqref{A:growth_gamma} at the two vertices $\pi'=(0,1)$ and $\pi'=(1,0)$ reads
\[
    e^{\,\mu_\Delta+\sigma_\Delta^2/2}>\bar s^{-1}(1+C)\qquad\text{and}\qquad e^{-\mu_\Delta+\sigma_\Delta^2/2}>\bar s^{-1}(1+C),
\]
or equivalently $\sigma_\Delta^2>2\ln\!\big(\bar s^{-1}(1+C)\big)+2|\mu_\Delta|$. In Example~\ref{ex:2d} we provide an explicit set of parameters, where this inequality is satisfied. 
\end{example}

Now, we numerically verify the convergence results stated in the paper, i.e. Theorem~\ref{th:gamma_asympt}, Theorem~\ref{th:gamma_asympt_strategy}, and Theorem~\ref{th:discrete_convergence}.

\begin{example}\label{ex:2d}
    In this example, we illustrate the vanishing gamma asymptotics results from Section~\ref{S:asymptotics}. To do this, we consider a two-asset portfolio with deterministic investment opportunities arrivals $\eta_i\equiv 1$. Logarithmic returns for the periods $[i,i+1]$ are i.i.d. and normally distributed with
    \[
    \begin{bmatrix}
        r_1(i, i+1) \\ r_2(i, i+1))
    \end{bmatrix}\sim \mathcal{N}_2\left(\mu,\Sigma \right)\quad \text{with}\quad \mu:= \begin{bmatrix}
        0.15 \\ 0.12
    \end{bmatrix},\, \Sigma:=\begin{bmatrix} 0.05 & -0.03 \\ -0.03 & 0.04 \end{bmatrix},
    \]
    and proportional transaction costs $c:=h:=[0.0005, 0.0005]$. It should be noted that the first asset has both a higher expected value and a higher variance of the log-returns. Thus, the risk-averse investor in the optimal strategy needs to balance return maximisation with risk control. 

    First, let us use Example~\ref{ex:lognormal} to verify the model assumptions. 
    In fact, using the notation from this example, we have $\mu_\Delta=0.03$, $\sigma_\Delta^2=0.15$, and $\bar c:=\max_{i=1,\ldots,d}\max(c_i,h_i)=0.0005$, so $\bar s^{-1}(1+C)\approx1.002$ and $2\ln\!\big(\bar s^{-1}(1+C)\big)+2|\mu_\Delta|\approx 0.064$. Thus, the bound is met with room to spare. In fact, a direct computation shows that~\eqref{A:growth_gamma} holds throughout the range $\gamma\in[-2,0]$ used in Figure~\ref{fig:ex2d}, the two vertices ratios in~\eqref{eq:growth_gamma} ranging over $[1.05,1.28]$ against a threshold of at most $1.004$.
    
    Using value iteration with 15 iterations on a discretized grid of portfolio weights, we compute the optimal ergodic portfolio strategies for risk-sensitive criterion~\eqref{eq:mu_log_return} with $\gamma \in \{ -2, -1.5, -1, -0.5, -0.25, -0.1, -0.01, 0\}$. As seen in the left panel of Figure~\ref{fig:ex2d}, the value of the problem increases monotonically to the optimal value of the risk-neutral problem as $\gamma$ increases to $0$. Additionally, we evaluate the ergodic value of the risk-neutral strategy when applied to the functional corresponding to different $\gamma$. As proved in Theorem~\ref{th:strategy_gamma_asympt}, the difference between the optimal value and the value corresponding to the risk-neutral strategy converges to $0$ as $\gamma$ converges to $0$, confirming that the risk-neutral strategy is almost optimal for the risk-averse problem when $\gamma$ is close to $0$. This is additionally justified in the right panel of Figure~\ref{fig:ex2d}, where we present the optimal strategy from Theorem~\ref{th:verification} for various $\gamma$. Note that, on the plot, it is enough to show only the first coordinates of the pre-rebalancing and post-rebalancing weights, as the problem is two-dimensional and the coordinates sum up to 1. As expected, the optimal strategies exhibit increasing conservatism when $\gamma$ becomes more negative, with a smaller proportion of capital invested in the first asset. When $\gamma$ converges to $0$, the optimal strategies converge to each other, as shown in Theorem~\ref{th:gamma_asympt_strategy}.

    \begin{figure}
        \centering
        \includegraphics[width=0.45\linewidth]{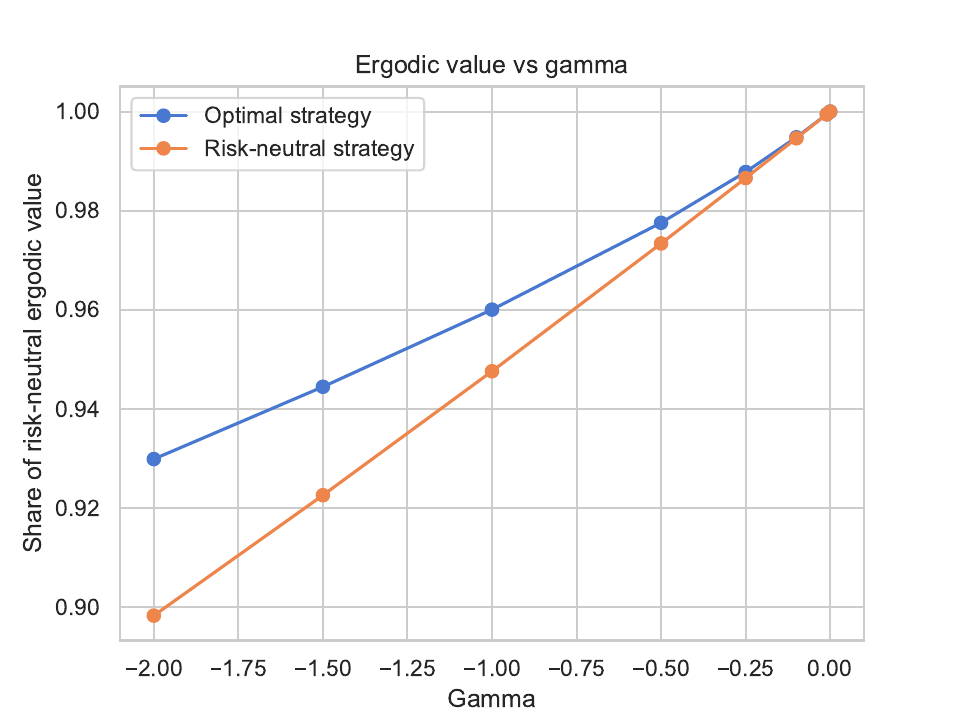}
        \includegraphics[width=0.45\linewidth]{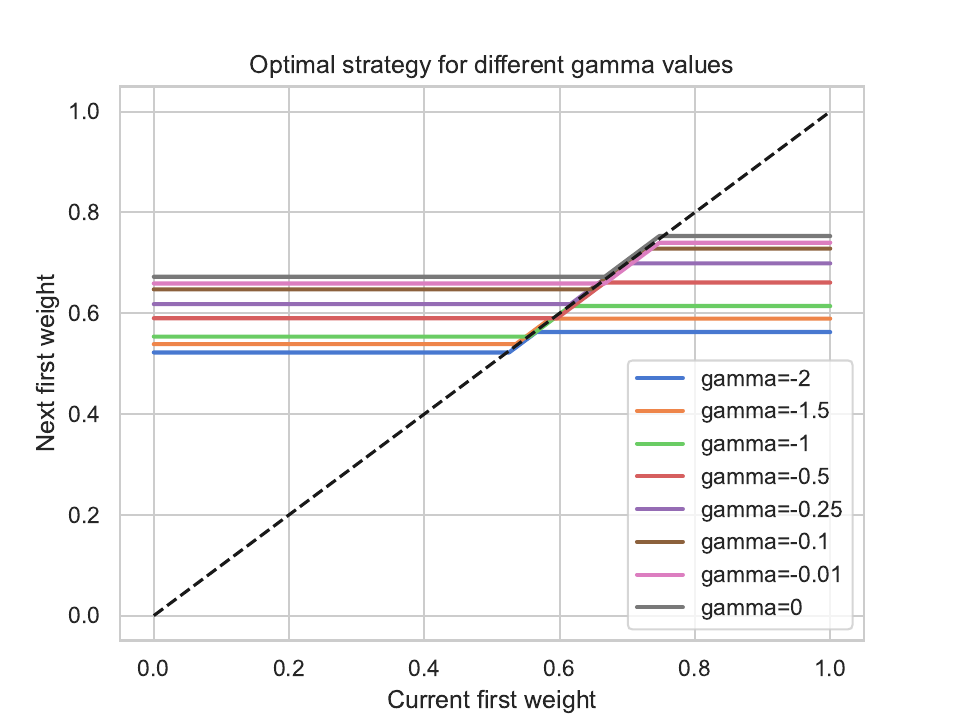}
        \caption{Optimal values and optimal strategies in Example~\ref{ex:2d}. The left panel shows the ratios of optimal values (blue dots) and values corresponding to the risk-neutral strategy (orange dots) of the underlying problems (for various $\gamma$) to the optimal value of the risk-neutral problem. The right panel shows the optimal strategies for different $\gamma$. On the X axis, there is the first coordinate $\pi_1$ of the pre-rebalancing weights $(\pi_1, 1-\pi_1)$ while on the $Y$ axis there is the first coordinate $\pi_1'$ of the optimal after-rebalancing weights $(\pi_1', 1-\pi_1')$.}
        \label{fig:ex2d}
    \end{figure}
\end{example}

\begin{example}\label{ex:discretisation}
    In this example, we validate the convergence of discretised problem from Section~\ref{S:discrete}. More specifically, we assume that $\eta_i$ are i.i.d. exponential random variables with the rate parameter $\lambda=1$. Also, for various $m$, we consider the discretised waiting times $\eta_i^m$ given by~\eqref{eq:eta_m}; note that in this setting $(\eta_i^m)_{i=1}^\infty$ are i.i.d. with geometric distribution on $\mathbb{N}_*$ and the probability of success $1-e^{-\lambda 2^{-m}}$. Next, we set $d=2$ and assume that the price process follows the bivariate Black-Scholes dynamics, or, equivalently, the log-returns satisfy
    \[
    \begin{bmatrix}
        r_1(t) \\ r_2(t))
    \end{bmatrix} =\mu t+\Sigma^{1/2} W_t,
    \]
    where $\mu$ and $\Sigma$ are as in Example~\ref{ex:2d}, and $(W_t)$ is the standard bivariate Brownian motion; note that this is simply a continuous-time version of the model studied in Example~\ref{ex:2d}. Finally, we set $\gamma=-2$.

    Using value iteration with 35 iterations on a discretized grid of portfolio weights, we compute the optimal value of the problem stated in~\eqref{eq:J(pi)_log:m} for $m\in \{1, 2, \ldots, 10\}$. As seen in the left panel of Figure~\ref{fig:ex:discretised}, the optimal value of the problem quickly converges, with relative error smaller than $0.001$ for $m=7$ (mesh grid $2^{-7}=1/128$). This is in line with Theorem~\ref{th:discrete_convergence}, which shows that the discretised problem converges to its continuous-time version. Also, the right panel of Figure~\ref{fig:ex:discretised} shows the expected value of the waiting times $\mathbb{E}[\eta_i^m]$ against $m$. Visible convergence is aligned with the fact that $\mathbb{E}[\eta_i^m]\to \mathbb{E}[\eta_i]=1/\lambda$ as $m\to\infty$.

        \begin{figure}
        \centering
        \includegraphics[width=0.45\linewidth]{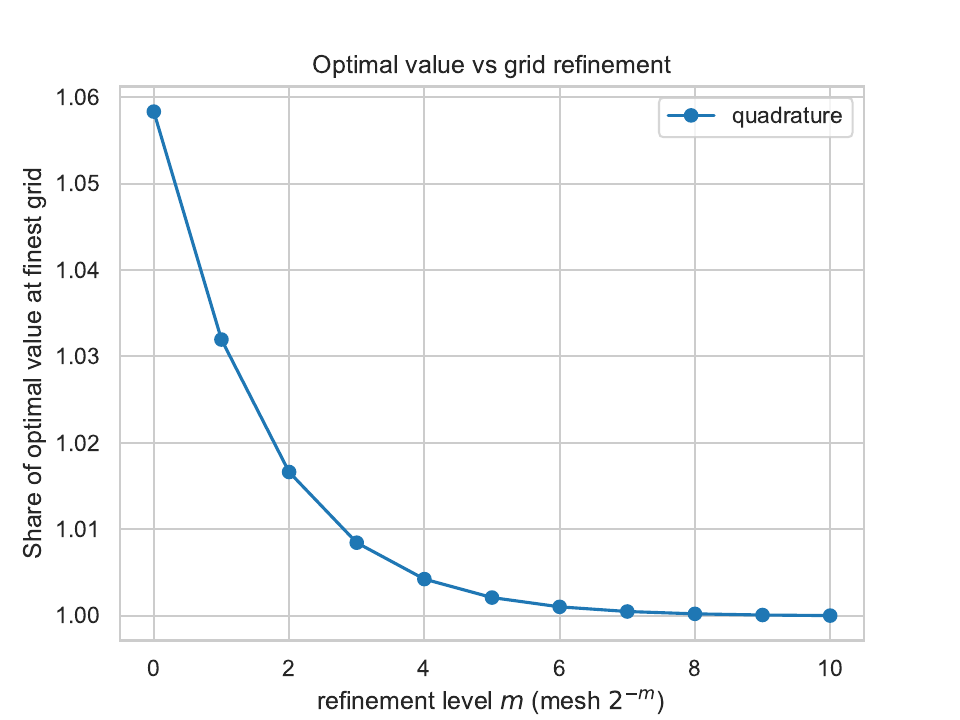}
        \includegraphics[width=0.45\linewidth]{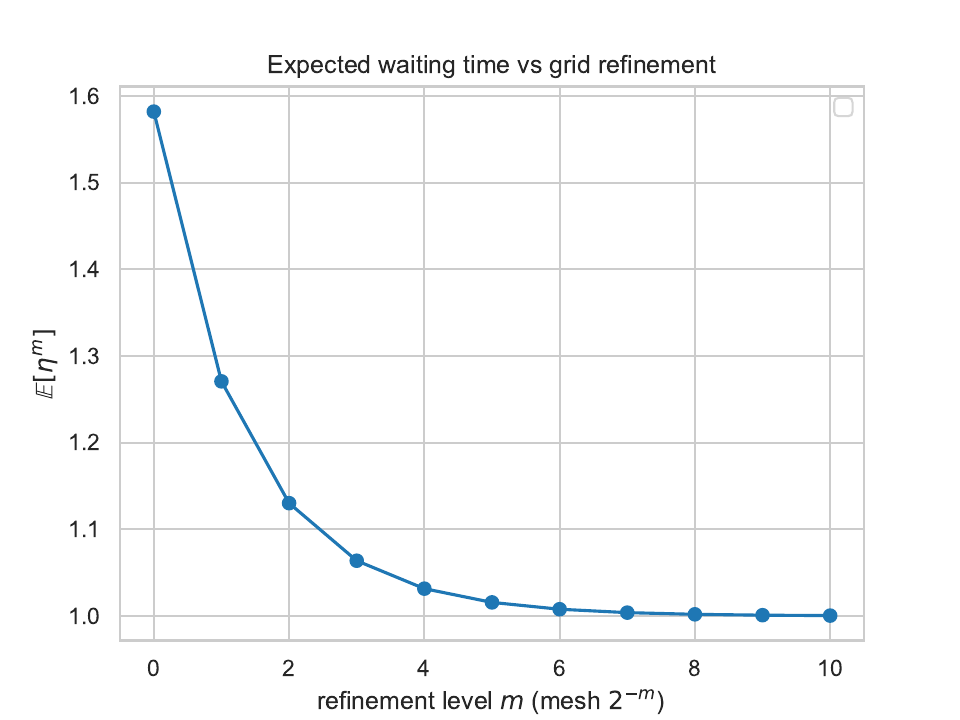}
        \caption{Optimal values and expected waiting times in Example~\ref{ex:discretisation}. The left panel shows the ratios of optimal values corresponding to different $m$ to the optimal value corresponding to $ m=10$. The right panel shows the expected value of the waiting time $\mathbb{E}[\eta_i^m]$ for different $m$.}
        \label{fig:ex:discretised}
    \end{figure}
    \end{example}

\section*{Declaration of generative AI use}
\noindent During the preparation of this work, the authors used large language model assistants to support proof idea generation, proofreading, and code writing for numerical examples. All AI-assisted output was used under supervision: the authors verified every derivation and numerical result and edited the text as needed. The authors take full responsibility for the content of the publication.

\section*{Funding}
Damian Jelito and {\L}ukasz Stettner acknowledge research support by Polish National Science Centre grant no. 2024/53/B/ST1/00703.

\appendix

\bibliographystyle{agsm}
\bibliography{RSC_bibliografia}

\end{document}